\pdfoutput=1
\documentclass{article}
\usepackage{amsmath, amsthm, amssymb}
\usepackage{eucal}
\usepackage{mathrsfs}
\usepackage{geometry}
\usepackage[unicode,bookmarksnumbered=true,colorlinks=true,allcolors=blue,linktoc=all]{hyperref}
\usepackage[capitalise,nameinlink]{cleveref}
\usepackage{tikz-cd}
\usepackage{stmaryrd}
\usepackage{listings}
\usepackage{etoolbox}
\usepackage{thm-restate}
\usepackage[nottoc,notlot,notlof]{tocbibind}
\usepackage{enumitem}
\usepackage{mathtools}
\usepackage{quiver}
\usepackage{caption}

\crefname{subsection}{Subsection}{subsections}

\hypersetup{bookmarksdepth=2}
\newtheorem{theorem}{Theorem}

\newtheorem{thm}{Theorem}[section]
\newtheorem{lem}[thm]{Lemma}
\newtheorem{prop}[thm]{Proposition}
\newtheorem{cor}[thm]{Corollary}
\newtheorem{conj}[thm]{Conjecture}

\theoremstyle{definition}
\newtheorem{defn}[thm]{Definition}
\newtheorem{notn}[thm]{Notation}
\newtheorem{const}[thm]{Construction}

\newtheorem{remark}[thm]{Remark}
\newtheorem{example}[thm]{Example}
\newtheorem{warn}[thm]{Warning}

\newcommand{\ncmd}{\newcommand}

\definecolor{DefColor}{rgb}{0.6,0.15,0.25}
\newcommand{\mdef}[1]{\textcolor{DefColor}{#1}}
\newcommand{\tdef}[1]{\mdef{\emph{#1}}}

\ncmd{\mbb}[1]{\mathbb{#1}}
\ncmd{\mrm}[1]{\mathrm{#1}}
\ncmd{\mcl}[1]{\mathcal{#1}}
\ncmd{\mfk}[1]{\mathfrak{#1}}
\ncmd{\mbf}[1]{\mathbf{#1}}
\ncmd{\mscr}[1]{\mathscr{#1}}

\ncmd{\todo}[1]{\textbf{TODO #1}}
\ncmd{\reftodo}[1]{\textbf{REF #1}}

\DeclareRobustCommand{\minwidthbox}[2]{%
  \mathmakebox[\ifdim#2<\width\width\else#2\fi]{#1}%
}
\ncmd{\too}[1][]{\xrightarrow{\minwidthbox{#1}{1em}}}
\ncmd{\iso}{\too[\smash{\raisebox{-0.5ex}{\ensuremath{\scriptstyle\sim}}}]}
\ncmd{\oot}[1][]{\xleftarrow{\minwidthbox{#1}{1em}}}
\ncmd{\hooktoo}[1][]{\xhookrightarrow{\minwidthbox{#1}{1em}}}
\ncmd{\adj}[1][]{\mathrel{\substack{\xrightarrow{\minwidthbox{#1}{1em}} \\[-.7ex] \xleftarrow{\minwidthbox{#1}{1em}}}}}

\ncmd{\qin}{\quad\in\quad}

\ncmd{\Id}{\mrm{Id}}
\ncmd{\tId}{\mbf{Id}}
\ncmd{\Nm}{\mathrm{Nm}}

\ncmd{\BB}{\mrm{B}}
\ncmd{\clB}{\mcl{B}}
\ncmd{\CC}{\mcl{C}}
\ncmd{\DD}{\mcl{D}}
\ncmd{\EE}{\mcl{E}}
\ncmd{\TT}{\mcl{T}}
\ncmd{\KK}{\mrm{K}}
\ncmd{\KU}{\mrm{KU}}
\ncmd{\tA}{\mbf{A}}
\ncmd{\tB}{\mbf{B}}
\ncmd{\cX}{\mcl{X}}
\ncmd{\cY}{\mcl{Y}}
\ncmd{\Kn}{\KK(n)}
\ncmd{\Knp}{\KK(n+1)}
\ncmd{\Kt}{\KK(t)}
\ncmd{\Ko}{\KK(1)}
\ncmd{\Tnp}{\mrm{T}(n+1)}
\ncmd{\rmE}{\mrm{E}}
\ncmd{\En}{\rmE_n}
\ncmd{\Enp}{\rmE_{n+1}}
\ncmd{\lsF}{\mscr{F}}
\ncmd{\lsG}{\mscr{G}}
\ncmd{\bbC}{\mbb{C}}
\ncmd{\GG}{\mbb{G}}
\ncmd{\NN}{\mbb{N}}
\ncmd{\ZZ}{\mbb{Z}}
\ncmd{\QQ}{\mbb{Q}}
\ncmd{\clQ}{\mcl{Q}}
\ncmd{\Qab}{\QQ(\zeta_\infty)}
\ncmd{\Zp}{\ZZ_p}
\ncmd{\Qp}{\QQ_p}
\ncmd{\Fp}{\mbb{F}_p}
\ncmd{\Fpbar}{\overline{\mbb{F}}_p}
\renewcommand{\SS}{\mbb{S}}
\ncmd{\rmP}{\mrm{P}}
\ncmd{\OO}{\mcl{O}}

\ncmd{\irchi}[2]{\raisebox{\depth/2}{$#1\chi$}}
\DeclareRobustCommand{\rchi}{{\mathpalette\irchi\relax}}
\ncmd{\ch}{\rchi}
\ncmd{\cch}{\widehat{\scalebox{1.15}{$\rchi$}}}

\ncmd{\fL}{L}
\ncmd{\fLLam}{\fL^\Lambda}
\ncmd{\LKn}{L_{\Kn}}
\ncmd{\LKnp}{L_{\Knp}}
\ncmd{\LKt}{L_{\Kt}}
\ncmd{\LKo}{L_{\Ko}}

\ncmd{\LL}{\mrm{L}}
\ncmd{\RR}{\mrm{R}}
\ncmd{\BC}{\mrm{BC}}
\ncmd{\psa}{\oplus}
\ncmd{\dbl}{\mrm{dbl}}
\ncmd{\op}{\mrm{op}}
\ncmd{\co}{2\text{-}\mrm{op}}
\ncmd{\lax}{\mathrm{lax}}
\ncmd{\oplax}{\mathrm{oplax}}
\ncmd{\ofin}{1\text{-}\mathrm{fin}}
\ncmd{\mfin}{m\text{-}\mathrm{fin}}
\ncmd{\pifin}{\pi\text{-}\mathrm{fin}}
\ncmd{\cpl}{\mrm{c}}
\ncmd{\cocpl}{\mrm{cc}}
\ncmd{\seg}{\mrm{seg}}
\ncmd{\et}{\mrm{\acute{e}t}}

\ncmd{\pt}{\mrm{pt}}
\ncmd{\Mod}{\mrm{Mod}}
\ncmd{\cMod}{\widehat{\Mod}\vphantom{\Mod}}
\ncmd{\Vect}{\mrm{Vect}}
\ncmd{\Ab}{\mrm{Ab}}
\ncmd{\Fin}{\mrm{Fin}}
\ncmd{\Spaces}{\mcl{S}}
\ncmd{\Spacespi}{\Spaces_{\pifin}}
\ncmd{\Spacesm}{\Spaces_{\mfin}}
\ncmd{\Spacesrelpi}{\Spaces^{\pifin}}
\ncmd{\Spaceso}{\Spaces_{\ofin}}
\ncmd{\Spacesoab}{\Spaceso^\mrm{ab}}
\ncmd{\OS}{\mcl{OS}}
\ncmd{\OSpi}{\mcl{OS}_{\pifin}}
\ncmd{\PP}{\mbb{P}}
\ncmd{\PDivor}{\PP\mathrm{Div}_{\mathrm{or}}}
\ncmd{\Sp}{\mrm{Sp}}
\ncmd{\SpKn}{\Sp_{\Kn}}
\ncmd{\SpKt}{\Sp_{\Kt}}
\ncmd{\LocSys}{\mrm{LocSys}}
\ncmd{\extLocSys}{\widehat{\LocSys}}
\ncmd{\Span}{\mrm{Span}}
\ncmd{\tSpan}{\mbf{Span}}
\ncmd{\Spano}{\Span_1}
\ncmd{\Spanoh}{\tSpan_{1\!\frac{1}{2}}}
\ncmd{\Spant}{\tSpan_2}
\ncmd{\Ucplcocpl}{\mbf{U}^{\cpl,\cocpl}}
\ncmd{\Ccplcocpl}{\CC^{\cpl,\cocpl}}
\ncmd{\Usa}{\mbf{U}^\psa}
\ncmd{\Cat}{\mrm{Cat}}
\ncmd{\tCat}{\mbf{Cat}}
\ncmd{\tCatt}{\mbf{Cat}_2}
\ncmd{\Catinfsa}{\Cat^{\oplus\text{-}\infty}}
\ncmd{\PrL}{\mrm{Pr}^\mrm{L}}
\ncmd{\Mfg}{\mcl{M}_\mrm{fg}}

\ncmd{\yon}{\text{\usefont{U}{min}{m}{n}\symbol{'110}}}
\DeclareFontFamily{U}{min}{}
\DeclareFontShape{U}{min}{m}{n}{<-> dmjhira}{}

\DeclareMathOperator{\Tr}{Tr}

\DeclareMathOperator{\ind}{Ind}
\DeclareMathOperator{\Map}{Map}
\DeclareMathOperator{\End}{End}
\DeclareMathOperator{\Hom}{Hom}

\DeclareMathOperator{\Fun}{Fun}
\DeclareMathOperator{\tFun}{\mbf{Fun}}

\DeclareMathOperator{\CAlg}{CAlg}
\DeclareMathOperator{\CMon}{CMon}
\ncmd{\CMoninf}{\CMon_\infty}
\DeclareMathOperator{\Split}{Split}
\DeclareMathOperator*{\colim}{colim}

\DeclareMathOperator{\PSh}{\mcl{P}}

\ncmd\noloc{%
  \nobreak
  \mspace{6mu plus 1mu}
  {:}
  \nonscript\mkern-\thinmuskip
  \mathpunct{}
  \mspace{2mu}
}

\title{Higher Semiadditivity in Transchromatic Homotopy Theory}
\author{Shay Ben-Moshe\thanks{Faculty of Mathematics and Computer Science, Weizmann Institute of Science, Israel.}}
\date{}

\begin{document}
	\maketitle
	
	\begin{abstract}
		We study the compatibility of higher semiadditivity across different chromatic heights.
		We prove that the categorified transchromatic character map assembles into a parameterized semiadditive functor, showing that it is higher semiadditive up to a free loop shift.
		By decategorification, this implies that the transchromatic character map is compatible with integration maps, generalizing the well-known formula for the character of an induced representation.
		Through a further decategorification process, we compute the higher semiadditive cardinalities of $\pi$-finite spaces at Lubin--Tate spectra.
	\end{abstract}

	\vspace{1em}

	\begin{figure}[ht!]
	  \centering
	  \includegraphics[width=140mm]{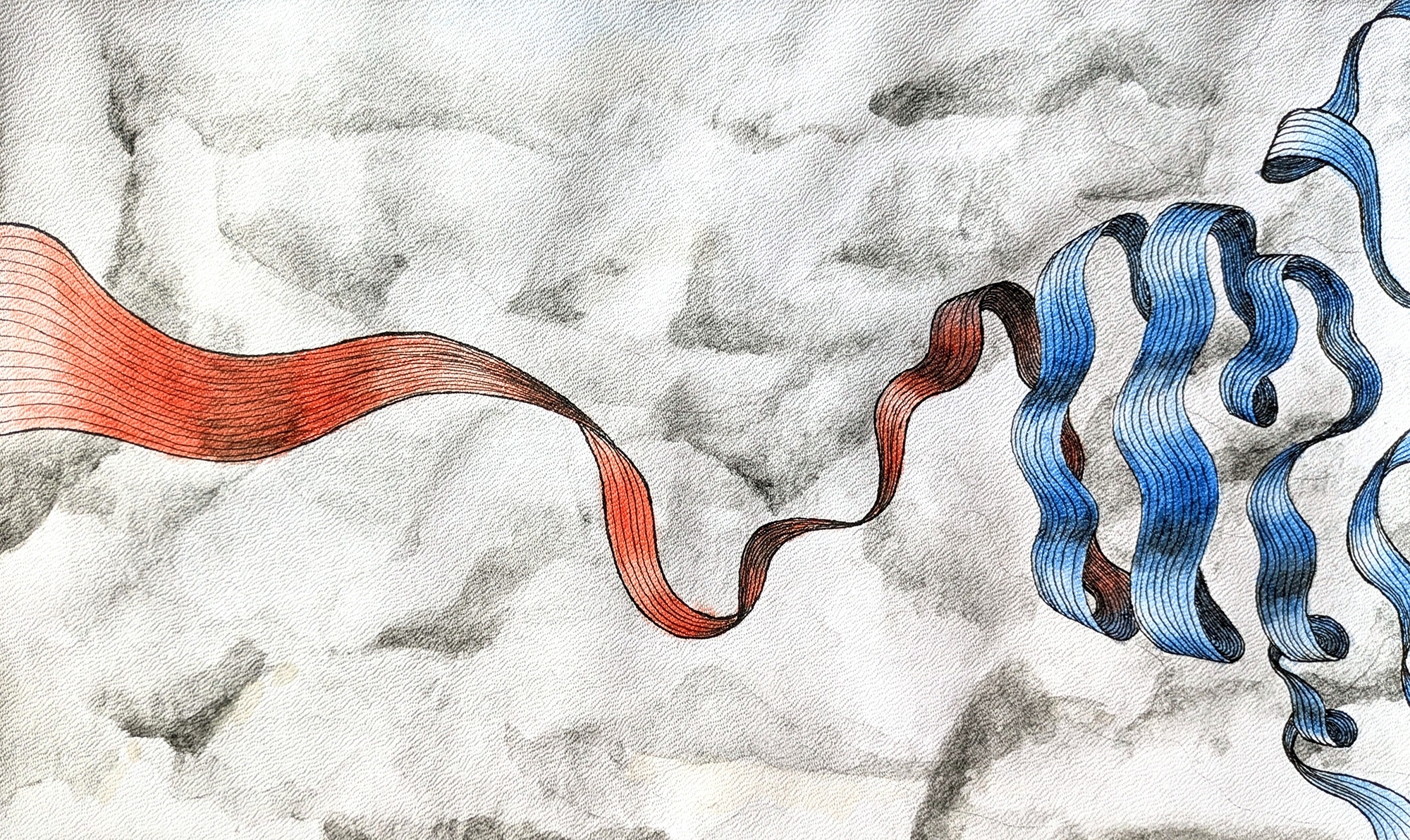}
	  \caption*{
		Untitled, by Irina Miodownik.
	  }
	\end{figure}

	\vspace{1em}

	\pagebreak
	
	\tableofcontents

	% \pagebreak
	
	\section{Introduction}

\subsection{Transchromatic Induced Characters}

Throughout this paper, we use the term category to mean an $(\infty,1)$-category and the term $2$-category to mean an $(\infty,2)$-category.

A main tool in representation theory is the character map
\[
	\ch\colon R(G) \too \Qab^{\fL \BB G}
\]
from the representation ring of a finite group $G$ to the ring of class functions (note that the free loop space $\fL \BB G$ is equivalent to the homotopy quotient $G/\mkern-5mu/G$ by the conjugation action).
% A notable property of the character map is that it induces an isomorphism
% \[
% 	\Qab \otimes R(G) \iso \Qab^{\fL \BB G}.
% \]
The Atiyah--Segal completion theorem \cite{AS} provides a homotopical interpretation of this construction, as it identifies the representation ring with the $0$-th homotopy group of $\KU^{\BB G}$, up to a completion.
Furthermore, note that the $0$-th homotopy group of the rationalization of $\KU$ is $\QQ$, which embeds into $\Qab$.
From this perspective, the character map relates chromatic heights $1$ and $0$.

The landmark works of Hopkins--Kuhn--Ravenel \cite{HKR}, Stapleton \cite{Stapleton} and Lurie \cite{Ell3} extend this construction to higher chromatic heights, as well as to more general spaces.
To start, recall that a formal group of a finite height $n \geq 0$ (over an implicit field of characteristic $p$) admits a universal deformation living over the Lubin--Tate spectrum $\En$, which in the case $n=1$ reproduces the $p$-completion of $\KU$.
Next, we would like to study the localization $\LKt\En$ at a lower height $0 \leq t < n$, in analogy with the rationalization of $\KU$.

While the formal group of $\En$ encapsulates many of its properties, as a formal scheme, the formal group is not well-behaved under base-change.
For example, its rationalization is a height $0$ formal group, losing all information.
On the other hand, the $p$-divisible group associated to the formal group, which over $\En$ captures the same information, is better behaved.
For instance, it retains its height under base-change.
This leads us to consider the base-change of the $p$-divisible group to $\LKt\En$.
Recall that any $p$-divisible group is canonically the extension of its \'etale part by its connected part.
We define the splitting algebra, denoted $C_t$, to be the initial commutative $\LKt\En$-algebra over which this extension splits, and the \'etale part becomes constant.
In the case $n=1$, the $0$-th homotopy group of $C_0$ is $\Qp(\zeta_{p^\infty})$, a $p$-adic analogue of $\Qab$.

With these in place, for any $\pi$-finite space $X$ (for instance, $X = \BB G$), there is a transchromatic character map
\[
	\ch\colon \En^X \too C_t^{\fL_p^{n-t} X},
\]
relating height $n$ and height $t$ information, where $\fL_p^{n-t} X := \Map(\BB \Zp^{n-t}, X)$ is the $(n-t)$-fold $p$-adic free loop space.
% Similarly to the ordinary character map, this induces an isomorphism
% \[
% 	C_t \otimes_{\En} \En^X \iso C_t^{\fL_p^{n-t} X}.
% \]

The ordinary character map is clearly contravariantly functorial in $G$, that is, it commutes with restriction of representations and class functions.
Importantly, it is also covariantly functorial, via the well-known formula for the character of an induced representation.
Indeed, for any inclusion of finite groups $f\colon H \hookrightarrow G$, the induction of representations and of class functions fit into a commutative square
% https://q.uiver.app/#q=WzAsNCxbMCwwLCJSKEgpIl0sWzEsMCwiXFxRYWJee1xcZkwgXFxCQiBIfSJdLFswLDEsIlIoRykiXSxbMSwxLCJcXFFhYl57XFxmTCBcXEJCIEd9LiJdLFswLDIsIlxcaW5kX2YiLDJdLFsxLDMsIlxcaW5kX3tcXGZMIGZ9Il0sWzAsMSwiXFxjaCJdLFsyLDMsIlxcY2giXV0=
\[\begin{tikzcd}[column sep=large]
	{R(H)} & {\Qab^{\fL \BB H}} \\
	{R(G)} & {\Qab^{\fL \BB G}.}
	\arrow["\ch", from=1-1, to=1-2]
	\arrow["{\ind_f}"', from=1-1, to=2-1]
	\arrow["{\ind_{\fL f}}", from=1-2, to=2-2]
	\arrow["\ch", from=2-1, to=2-2]
\end{tikzcd}\]
As the notation suggests, one can directly verify that the induction of characters is given by summation along the homotopy fibers of $\fL f\colon \fL \BB H \to \fL \BB G$, which are discrete and finite.

Recall that cohomology admits covariant transfer maps along maps with finite discrete fibers, endowing $\En$ with transfer maps along $\BB H \to \BB G$.
Hopkins--Kuhn--Ravenel proved a formula for the transchromatic character of these transfers (in the case $t=0$), which can similarly be interpreted as summation along the homotopy fibers of $\fL_p^n f\colon \fL_p^n \BB H \to \fL_p^n \BB G$.

In their seminal work, Hopkins--Lurie \cite{HL} proved that the $\Kn$-local categories are $\infty$-semiadditive.
Notably, this endows every $\Kn$-local spectrum with integration maps along maps of $\pi$-finite spaces, generalizing the aforementioned transfer maps.
As $\En$ is $\Kn$-local and, similarly, $C_t$ is $\Kt$-local, we may ask whether the transchromatic character map is compatible with these integration maps.
In \cite[Remark 7.4.8 and Remark 7.4.9]{Ell3}, Lurie indicates that this is indeed the case, in the sense that for a map of $\pi$-finite spaces $f\colon X \to Y$ there is a commutative square
% https://q.uiver.app/#q=WzAsNCxbMCwwLCJcXEVuXlgiXSxbMSwwLCJDX3Ree1xcZkxfcF57bi10fSBYfSJdLFswLDEsIlxcRW5eWSJdLFsxLDEsIkNfdF57XFxmTF9wXntuLXR9IFl9LiJdLFswLDIsIlxcaW50X2YiLDJdLFsxLDMsIlxcaW50X3tcXGZMX3Bee24tdH0gZn0iXSxbMCwxLCJcXGNoIl0sWzIsMywiXFxjaCJdXQ==
\[\begin{tikzcd}[column sep=large]
	{\En^X} & {C_t^{\fL_p^{n-t} X}} \\
	{\En^Y} & {C_t^{\fL_p^{n-t} Y}.}
	\arrow["\ch", from=1-1, to=1-2]
	\arrow["{\int_f}"', from=1-1, to=2-1]
	\arrow["{\int_{\fL_p^{n-t} f}}", from=1-2, to=2-2]
	\arrow["\ch", from=2-1, to=2-2]
\end{tikzcd}\]
We note that this extends the Hopkins--Kuhn--Ravenel formula even when $f$ is a map of classifying spaces and $t=0$, when $f$ does not come from an inclusion of groups (in which case the homotopy fibers are not discrete, and higher semiadditivity is needed).
For instance, the case $\BB G \to \pt$ provides a generalization of the well-known formula expressing the dimension of the fixed points of a representation as the average value of its character.

In this paper, we prove the compatibility of the transchromatic character map with integration in a highly coherent fashion, a question hinted in \cite[Remark 7.4.7]{Ell3}.
Recall that in a semiadditive category, every object canonically promotes into a commutative monoid, encoded by a functor from the category $\Spano(\Fin)$ of spans of finite sets, reflecting the distributivity of transfers over restriction.
Harpaz \cite{Harpaz} proved an analogue of this result for $\infty$-semiadditive categories.
More specifically, he showed that the category $\Spano(\Spacespi)$ of spans of $\pi$-finite spaces is the free $\infty$-semiadditive category.
As such, every object in an $\infty$-semiadditive category promotes to an $\infty$-commutative monoid, encoded by a functor from $\Spano(\Spacespi)$, which captures the distributivity of integration over restriction.
In particular, we may apply this to $\En \in \SpKn$, and to $C_t \in \SpKt$ and pre-compose with the free loop space functor, to obtain
\[
	\En^{(-)}\colon \Spano(\Spacespi) \too \SpKn,
	\qquad C_t^{\fL_p^{n-t}(-)}\colon \Spano(\Spacespi) \too \SpKt.
\]
With this in mind, our main result is the following.

\begin{theorem}[{\cref{ind-char-form}}]\label{ind-char-form-intro}
	The transchromatic character map
	\[
		\ch\colon \En^X \too C_t^{\fL_p^{n-t} X}
	\]
	extends to a natural transformation in $X \in \Spano(\Spacespi)$.
\end{theorem}

\subsection{Chromatic Cardinalities}\label{subsec-chrom-card}

Let $R$ be a $\Kn$-local commutative ring spectrum and let $X$ be a $\pi$-finite space.
Carmeli--Schlank--Yanovski \cite{AmbiHeight} defined the cardinality of $X$ at $R$ to be the element $|X|_R \in \pi_0(R)$ corresponding to the composition of restriction and integration along $f_X\colon X \to \pt$
\[
	R \too R^X \too[\int_{f_X}] R.
\]
For example, the cardinality $|X|_{\QQ}$ is the Baez--Dolan homotopy cardinality \cite{hocard}.

In our previous joint work with Carmeli, Schlank and Yanovski \cite{card}, we computed the cardinalities at $\En$, when $X$ is further assumed to be a $p$-space.
We refer the reader to the introduction of that paper for a more detailed exposition of this problem.
As was mentioned in \cite[Example 2.2.4]{AmbiHeight}, there is an alternative approach via the transchromatic character map, which applies to all $\pi$-finite spaces (see \cref{other-work} for a discussion of the relation between the two approaches).
Indeed, the compatibility of the transchromatic character map with restriction and integration immediately implies the following, computing chromatic cardinalities in terms of Baez--Dolan homotopy cardinalities.

\begin{theorem}[{\cref{chrom-card}}]\label{chrom-card-intro}
    For any $\pi$-finite space $X$, the cardinality $|X|_{\En}$ lands in $\ZZ_{(p)} \subset \pi_0(\En)$, and we have
    \[
        |X|_{\En} = |\fL_p^n X|_\QQ.
    \]
\end{theorem}

\subsection{The Categorified Transchromatic Character Map}

Our proof of \cref{ind-char-form-intro} is based on a categorified result, concerning the compatibility of higher semiadditivity across different chromatic heights through Lurie's categorified transchromatic character map, which is of independent interest.

One way to construct an isomorphism between two higher commutative monoids is via categorification:
given a functor between $\infty$-semiadditive categories $F\colon \CC \to \DD$ which is $\infty$-semiadditive, i.e.\ preserves $\pi$-finite (co)limits, and an object $M \in \CC$, we get an isomorphism
\[
	F(M^X) \iso F(M)^X,
\]
natural in $X \in \Spano(\Spacespi)$.
However, the transchromatic character map does not arise in this way, as is evident from the appearance of the free loop in the target.
Indeed, the functors
\[
	\LKt\colon \SpKn \too \SpKt,
\]
controlling the gluing of the different chromatic layers, are not $\infty$-semiadditive, nor are the base-changes to the simpler algebras $\En$ and $C_t$
\[
	C_t \otimes_{\LKt\En} \LKt(-)\colon \cMod_{\En} \too \cMod_{C_t}.
\]
Instead, Lurie's construction of the transchromatic character map is based on a more elaborate categorification, as we now recall.

Given a complex periodic\footnote{That is, $R$ is complex orientable, and weakly $2$-periodic, i.e.\ $\pi_2(R)$ is a projective module of rank $1$ over $\pi_0(R)$ and $\pi_2(R) \otimes_{\pi_0(R)} \pi_{-2}(R) \to \pi_0(R)$ is an isomorphism.} commutative ring spectrum $R$, a $p$-divisible group $\GG$ over $R$, and a space $X$, Lurie constructs a category $\LocSys_\GG(X)$ of tempered local systems on $X$, which is an enhancement of the category $\Mod_R^X$ of ordinary local systems with values in $R$-module spectra, suited for capturing transchromatic information.
When $R$ is $\Kn$-local and $\GG = \GG_R^Q$ is its Quillen $p$-divisible group, the subcategory $\LocSys_\GG^{\Kn}(X)$ on those tempered local systems that are $\Kn$-local is equivalent to the subcategory $\cMod_R^X$ on those local systems that are $\Kn$-local.
Through an analysis of the case where the connected--\'etale sequence of $\GG$ splits and the \'etale part is constant (as is the case for the base-change of $\GG_{\En}^Q$ to $C_t$), Lurie constructs the categorified transchromatic character map, which under the connection with ordinary local systems corresponds to a functor
\[
	\cch\colon \cMod_{\En}^X \too \cMod_{C_t}^{\fL_p^{n-t} X}.
\]
The last main ingredient in the construction of the transchromatic character map is Lurie's seminal tempered ambidexterity theorem.
This is most naturally phrased in the language of parameterized categories, which we now briefly recall.

The framework of parameterized category theory was developed by various groups, including Barwick--Dotto--Glasman--Nardin--Shah \cite{BDGNS} and Martini--Wolf \cite{MW}.
Of particular relevance to us is the notion of parameterized semiadditivity, as developed by Cnossen--Lenz--Linskens \cite{CLL}, extending the works of Hopkins--Lurie and Harpaz.
A $\Spacespi$-parameterized category is simply a functor $\CC\colon \Spacespi^\op \to \Cat$.
For instance, given any (non-parameterized) category $\CC$, we can define the canonically $\Spacespi$-parameterized category sending $X$ to $\CC^X$, and $f\colon X \to Y$ to the pullback functor $f^*\colon \CC^Y \to \CC^X$.
Continuing with this example, we see that if $\CC$ has $\pi$-finite colimits, then the functors $f^*$ admit left adjoints, given by computing colimits fiber-wise.
Generalizing this, we say that a $\Spacespi$-parameterized category $\CC\colon \Spacespi^\op \to \Cat$ is $\Spacespi$-cocomplete if all of the functors $f^*$ admit left adjoints, and the Beck--Chevalley condition holds for any pullback square in $\Spacespi$, and we dually define $\Spacespi$-completeness.
$\Spacespi$-semiadditivity is then defined by requiring the invertibility of the norm maps, whose construction extends essentially as-is to the parameterized setting.
In a similar fashion, we can study parameterized maps between parameterized categories, as well as properties such as being parameterized (co)continuous and parameterized semiadditive.

With this in mind, Lurie's tempered ambidexterity theorem says that the categories of tempered local systems $\LocSys_\GG(X)$ assemble into a $\Spacespi$-parameterized $\Spacespi$-semiadditive category $\LocSys_\GG$.
In other words, the higher semiadditivity of $\Kn$-local $R$-modules extends to all $R$-modules, when they are suitably tempered.
Using this, we prove our main theorem, which is the analogues result for the categorified transchromatic character map, relating the higher semiadditivity of the different chromatic heights in the more familiar context of ordinary local systems.

\begin{theorem}[{\cref{main-thm}}]\label{main-thm-intro}
	The categorified transchromatic character map assembles into a $\Spacespi$-semiadditive $\Spacespi$-parameterized map between $\Spacespi$-semiadditive $\Spacespi$-parameterized categories
	\[
		\cch\colon \cMod_{\En}^{(-)} \too \cMod_{C_t}^{\fL_p^{n-t}(-)}.
	\]
\end{theorem}

Briefly, our proof proceeds as follows.
The construction of the categorified transchromatic character map is as a composition of a series of functors between categories of tempered local systems.
We first enhance each of these functors to a $\Spacespi$-parameterized map.
Now, recall that a (non-parameterized) left adjoint functor is colimit preserving.
Hence, a left adjoint functor between $\infty$-semiadditive categories is $\infty$-semiadditive, and the same holds dually for right adjoints.
Essentially by the same argument, this also holds in the parameterized setting.
Lurie's tempered ambidexterity theorem guarantees that all $\Spacespi$-parameterized categories involved are $\Spacespi$-semiadditive.
Thus, to conclude, we show that each of the $\Spacespi$-parameterized maps is either a left or a right adjoint, and hence $\Spacespi$-semiadditive, whence so is their composition.

Turning back to the construction of the transchromatic character map, \cref{main-thm-intro} allows us to reformulate Lurie's decategorification process in the following way.
Letting $f_X\colon X \to \pt$ and $f_{\fL_p^{n-t}X}\colon \fL_p^{n-t}X \to \pt$ denote the unique maps, we have
\[
	C_t \otimes_{\En} \En^X
	\simeq \cch((f_X)_* f_X^* \En)
	\simeq (f_{\fL_p^{n-t}X})_* f_{\fL_p^{n-t}X}^* \cch(\En)
	\simeq (f_{\fL_p^{n-t}X})_* f_{\fL_p^{n-t}X}^* C_t
	\simeq C_t^{\fL_p^{n-t}X},
\]
where the second isomorphism is by \cref{main-thm-intro}.

By controlling this decategorification process more carefully, we also prove \cref{ind-char-form-intro}.
Cnossen--Lenz--Linskens \cite{CLL} recently extended Harpaz's description of the free $\infty$-semiadditive category to the parameterized setting.
In particular, given an object in the underlying category of a $\Spacespi$-semiadditive $\Spacespi$-parameterized category $M \in \CC_\pt$, we get a parameterized analogue of a higher commutative monoid
\[
	M_\CC^{(-)}\colon \Spano(\Spacespi) \too \CC_\pt,
\]
sending a $\pi$-finite space $X$ to $M_\CC^X := f_{X*} f_X^* M$, and with functoriality given by the parameterized analogues of restriction and integration.
This construction is manifestly functorial in parameterized semiadditive maps, in the sense that given a $\Spacespi$-semiadditive $\Spacespi$-parameterized map $F\colon \CC \to \DD$, we get an isomorphism
\[
	F(M_\CC^X) \iso F(M)_\DD^X,
\]
natural in $X \in \Spano(\Spacespi)$.
\cref{ind-char-form-intro} thus follows by applying this to the categorified transchromatic character map.

The subtlety in this argument is the study of the parameterized integration maps.
Specifically, we need to identify the parameterized integration maps in the parameterized category $\cMod_{C_t}^{\fL_p^{n-t}(-)}$ with the $(n-t)$-fold $p$-adic free loop space of the integration maps in $\cMod_{C_t}^{(-)}$, in a coherent way.
We establish this in \cref{G-sa-cohint} by employing the $2$-category of iterated spans $\Spant(\Spacespi)$, and its connection to parameterized semiadditivity.
More precisely, we generalize the theory of parameterized categories to parameterized objects in any $2$-category, and show that $\Spant(\Spacespi)$ carries a canonical parameterized semiadditive object (\cref{span-2-sa}).
In fact, we conjecture that it is the universal $2$-category with a parameterized object (\cref{span-2-univ}).

\subsection{Relation to Other Work}\label{other-work}

\subsubsection{Categorical Characters}

Let $R$ be a $\Kn$-local commutative ring spectrum.
Given a local system of dualizable modules $M\colon X \to \cMod_R^\dbl$ indexed on a space $X$, there is an associated character
\[
	\ch_M\colon \fL X \too R,
	\qquad \ch_M(\gamma) := \Tr(\gamma \mid M).
\]
This construction assembles into the categorical character map
\[
	\ch\colon \cMod_R^{\dbl,X} \too R^{\fL X}.
\]
Taking $R = \bbC$ and $X = \BB G$ reproduces the ordinary character map.

Note that the category $\cMod_R^\dbl$ is $\infty$-semiadditive.
By the work of Harpaz \cite{Harpaz}, the category of $\infty$-semiadditive categories is itself $\infty$-semiadditive.
As such, $\cMod_R^\dbl$ admits integration maps along maps of $\pi$-finite spaces, which are given by Kan extensions, and thus generalizes the usual induction of representations when $R = \bbC$.
In work in progress by Carmeli--Cnossen--Ramzi--Yanovski, following up on their previous work \cite{CatChar}, it is shown that the categorical character map is compatible with integration for any $\Kn$-local commutative ring spectrum $R$, generalizing the case of height $0$, in analogy with \cref{ind-char-form-intro}.
As an immediate consequence, the categorical character map takes the cardinality of $X$ at $\cMod_R^\dbl$ to the cardinality of $\fL X$ at $R$, that is
\[
	\dim(|X|_{\cMod_R^\dbl}) = |\fL X|_R.
\]
This corollary was already proven independently by Carmeli--Schlank--Yanovski in \cite{AmbiHeight}.

The categorical setting can be bridged back to the chromatic setting via chromatically localized algebraic K-theory.
Starting again with the case $R = \bbC$ and $X = \BB G$ where $G$ is a finite $p$-group, the Atiyah--Segal completion theorem shows that the $\Ko$-localized algebraic K-theory is
\[
	\LKo\KK(\cMod_\bbC^{\dbl,\BB G})
	\simeq \LKo\KK(\cMod_\bbC^{\dbl})^{\BB G}
	\simeq \KU_p^{\BB G}.
\]
Furthermore, in a joint work with Schlank \cite{BMS} we have shown that this construction is compatible with integration maps, so that the integration maps on $\KU_p$ correspond to the induction of representations under the Atiyah--Segal isomorphism.
In our joint work with Carmeli, Schlank and Yanovski \cite{Desc} we extended this to all heights, showing in particular that if $X$ is a $\pi$-finite $p$-space, then
\[
	\LKnp\KK(\cMod_R^{\dbl,X})
	\simeq \LKnp\KK(\cMod_R^\dbl)^X
	\simeq \LKnp\KK(R)^X,
\]
and similarly, this construction is compatible with integration maps, hence with cardinalities.

As we mentioned above, in \cite{card} we gave a different proof of \cref{chrom-card-intro} in the special case where $X$ is a $p$-space.
Briefly put, combining the two compatibilities with cardinalities described above, we deduce that
\[
	|X|_{\LKnp\KK(\En)} = |\fL X|_{\En}.
\]
Yuan \cite{yuan} has shown that $\LKnp\KK(\En)$ does not vanish, and thus by the chromatic nullstellensatz of Burklund--Schlank--Yanovski \cite{null}, it admits a map to a height $n+1$ Lubin--Tate spectrum.
This implies that
\[
	|X|_{\Enp} = |X|_{\LKnp\KK(\En)} = |\fL X|_{\En},
\]
which in turn implies \cref{chrom-card-intro} for $p$-spaces by induction.

% Aside from its relation to the present discussion, we used this to prove an example of hyperdescent for $\Knp$-localized algebraic K-theory.
% In their groundbreaking work, Burklund--Hahn--Levy--Schlank \cite{tel} gave counterexamples to the $\Tnp$-localized analogue, thereby disproving Ravenel's long-standing telescope conjecture.

% As this line of reasoning indicates, the categorical character map, or rather the Dennis trace map which bridges back to the chromatic setting, seems to be compatible with the transchromatic character map, through the chromatic nullstellensatz.

\subsubsection{Twisted Semiadditivity}

\cref{main-thm-intro} establishes the compatibility of the $\infty$-semiadditivity of different chromatic layers, with an introduction of a free loop shift.
However, this result only concerns this compatibility over the faithfully flat $\SS_{\Kn}$-algebra $\En$ and the faithfully flat $\LKt\En$-algebra $C_t$.
This leaves open the question of descending back to the $\Kn$- and $\Kt$-local spheres, to directly study the functor
\[
	\LKt\colon \SpKn \too \SpKt.
\]
This would involve two different descent processes: from $C_t$ to $\LKt\En$, and from $\En$ to $\SS_{\Kn}$.
As an illustration for the complications that arise in this case, we note that in contrast to the case of $\En$, the cardinalities over the $\Kn$-local sphere are not $p$-adic integers.
For example, Carmeli--Yuan \cite{CY} proved that at height $n=1$ and the prime $p=2$, we have
\[
	|\BB C_2|_{\SS_{\Ko}} = 1 + \varepsilon
	\qin \pi_0(\SS_{\Ko}) \simeq \ZZ_2[\varepsilon]/(\varepsilon^2, 2\varepsilon).
\]
This question, particularly for the case $t=0$, is closely related to the recent calculation of the rationalization of the $\Kn$-local sphere by Barthel--Schlank--Stapleton--Weinstein \cite{RatKn}.

\subsection{Organization}

For the reader's convenience, we provide a brief linear overview of the contents of the paper.

In \cref{sec-param}, we give an overview of the theory of parameterized categories with an emphasis on parameterized semiadditivity, and extend it to parameterized objects in $2$-categories other than the $2$-category of categories.
Particularly, we prove \cref{span-2-sa}, showing that the $2$-category of iterated spans carries a parameterized semiadditive object, and present \cref{span-2-univ} stating that it is the universal example.

In \cref{sec-decat}, we study the decategorification of parameterized semiadditive parameterized categories, and maps between them.
We use the description of the free parameterized semiadditive parameterized category of Cnossen--Lenz--Linskens to construct the parameterized analogue of cohomology, together with its integration maps, in a coherent fashion.
We describe a further decategorification process, constructing the parameterized analogue of higher semiadditive cardinalities.
Finally, we employ \cref{span-2-sa} mentioned above to relate these constructions to the non-parameterized setting.

In \cref{sec-temp}, we show that the categories of tempered local systems assemble into a parameterized category, and that the various functors between them constructed by Lurie assemble into parameterized adjoint maps.

Finally, in \cref{sec-ind-char}, we combine the results of \cref{sec-param} and \cref{sec-temp} to prove \cref{main-thm-intro}, and apply the decategorification processes from \cref{sec-decat} to deduce \cref{ind-char-form-intro} and \cref{chrom-card-intro}.

\subsection{Acknowledgements}

I am deeply indebted to Lior Yanovski for his contributions to this work.
Notably, he introduced me to the question of the compatibility of the transchromatic character map with integration, and explained the idea that morphisms of cohomologies shifted by free loop spaces should arise from suitable categorified maps.
I am grateful for his guidance and support throughout the project, and for many helpful discussions and comments on the manuscript.
I would also like to thank Tomer Schlank for his support and interest in the project, and Emmanuel Farjoun for his comments on an earlier draft of this paper.
Finally, I would like to thank Irina Miodownik for painting the artwork on the front page.

	\section{Parameterized Category Theory}\label{sec-param}

In this section we review and extend the theory of parameterized category theory, mostly following Martini--Wolf \cite{MW}, and Cnossen--Lenz--Linskens \cite{CLL} who developed the theory of parameterized semiadditivity.
We differ from these references in two important ways:
\begin{enumerate}
    \item We do not parameterize over an arbitrary topos, rather, our case is equivalent to parameterizing over the presheaf topos $\PSh(\cX)$ where $\cX$ is some category.
    \item We extend the theory from parameterized categories to parameterized objects in any $2$-category $\tA$.
    While the most important case is indeed $\tA = \tCat$, this flexibility allow us to phrase the universal property of the $2$-category of spans in this language (see \cref{span-1.5-univ}), and our related study of the $2$-category of iterated spans (see \cref{span-2-sa} and \cref{span-2-univ}).
\end{enumerate}
With this in mind, we fix the following.

\begin{notn}
    Throughout this section, $\mdef{\cX}$ is a category with pullbacks, and $\mdef{\tA}$ is a $2$-category.
\end{notn}

\subsection{Parameterized Objects and Adjunctions}

\begin{defn}
    We define the $2$-category of \tdef{$\cX$-parameterized $\tA$-objects} to be the $2$-category of functors $\mdef{\tFun(\cX^\op, \tA)}$.
    For an $\cX$-parameterized $\tA$-object $\CC$ and $f\colon X \to Y$ in $\cX$, we denote the image of $f$ under $\CC$ by $f^*\colon \CC_Y \to \CC_X$.
\end{defn}

This naming convention is motivated by the following example.

\begin{example}\label{can-param}
    For $\tA = \tCat$, the $2$-category $\tFun(\cX^\op, \tCat)$ is the $2$-category of \tdef{$\cX$-parameterized categories}.

    If furthermore $\cX \subset \Spaces$ is a subcategory of spaces, any (non-parameterized) category $\CC \in \Cat$ gives rise to the \tdef{canonically $\cX$-parameterized category} $\CC\colon \cX^\op \to \tCat$, given by the restriction of the functor $\CC^{(-)}\colon \Spaces^\op \to \tCat$ to $\cX^\op$.
    Namely, the functor sending $X \in \cX$ to $\CC_X := \CC^X$ and $f\colon X \to Y$ to the pre-composition with $f$ functor $f^* \colon \CC_Y \to \CC_X$.
    (The reader may wonder why we have restricted $\cX$ to be a subcategory of spaces rather than of categories. We return to this point in \cref{warn-cocomp} below.)
\end{example}

As $\tFun(\cX^\op, \tA)$ is a $2$-category, one may consider adjoint $1$-morphisms.
These turn out to admit a simple description.

\begin{prop}\label{param-adj}
    A $1$-morphism $L\colon \CC \to \DD$ between $\cX$-parameterized $\tA$-objects is a left adjoint in $\tFun(\cX^\op, \tA)$ if and only if the following two conditions hold:
    \begin{enumerate}
        \item For any $X \in \cX$, the $1$-morphism $L_X\colon \CC_X \to \DD_X$ in $\tA$ has a right adjoint $R_X\colon \DD_X \to \CC_X$.
        \item For any $f\colon X \to Y$ in $\cX$ the Beck--Chevalley map
            \[
                f^* R_Y \too R_X f^*
            \]
            obtained from taking vertical right adjoints in the commutative square
            % https://q.uiver.app/#q=WzAsNCxbMSwwLCJcXENDKFgpIl0sWzAsMCwiXFxDQyhZKSJdLFswLDEsIlxcREQoWSkiXSxbMSwxLCJcXEREKFgpIl0sWzEsMCwiZl4qIl0sWzIsMywiZl4qIl0sWzEsMiwiTF9ZIiwyXSxbMCwzLCJMX1giXV0=
            \[\begin{tikzcd}
                {\CC_Y} & {\CC_X} \\
                {\DD_Y} & {\DD_X}
                \arrow["{f^*}", from=1-1, to=1-2]
                \arrow["{f^*}", from=2-1, to=2-2]
                \arrow["{L_Y}"', from=1-1, to=2-1]
                \arrow["{L_X}", from=1-2, to=2-2]
            \end{tikzcd}\]
            is an isomorphism.
    \end{enumerate}
\end{prop}

\begin{proof}
    This is (a special case of) \cite[Corollary 5.2.12(2)]{fun-adj}.
    The case of parameterized categories (i.e.\ $\tA = \tCat$) was previously proven directly in \cite[Proposition 3.2.9 and Corollary 3.2.11]{MW}.
\end{proof}

\subsection{Parameterized (Co)completeness}

\subsubsection{Definitions and Basic Properties}

Let $\CC$ be a (non-parameterized) category, and consider the canonically parameterized category $\CC\colon \Spaces^\op \to \tCat$ of \cref{can-param}.
Recall that $\CC$ admits colimits indexed by spaces if and only if for any $f\colon X \to Y$, the functor $f^*\colon \CC^Y \to \CC^X$ admits a left adjoint $f_!\colon \CC^X \to \CC^Y$, which is then given by the left Kan extension along $f$.
Furthermore, in this case, left Kan extensions are given by the colimit along the fibers.
We turn these observations into definitions in the general case.

\begin{defn}
    An $\cX$-parameterized $\tA$-object $\CC\colon \cX^\op \to \tA$ is called \tdef{$\cX$-cocomplete} if
    \begin{enumerate}
        \item for any $f\colon X \to Y$ in $\cX$, the morphism $f^*\colon \CC_Y \to \CC_X$ admits a left adjoint $f_!\colon \CC_X \to \CC_Y$,
        \item and for every pullback square
            % https://q.uiver.app/#q=WzAsNCxbMCwwLCJQIl0sWzEsMCwiWSJdLFswLDEsIlgiXSxbMSwxLCJaIl0sWzAsMSwiXFx3aWRldGlsZGV7Zn0iXSxbMSwzLCJnIl0sWzIsMywiZiJdLFswLDIsIlxcd2lkZXRpbGRle2d9IiwyXSxbMCwzLCIiLDEseyJzdHlsZSI6eyJuYW1lIjoiY29ybmVyIn19XV0=
            \[\begin{tikzcd}
                P & Y \\
                X & Z
                \arrow["{\widetilde{f}}", from=1-1, to=1-2]
                \arrow["g", from=1-2, to=2-2]
                \arrow["f", from=2-1, to=2-2]
                \arrow["{\widetilde{g}}"', from=1-1, to=2-1]
                \arrow["\lrcorner"{anchor=center, pos=0.125}, draw=none, from=1-1, to=2-2]
            \end{tikzcd}\]
            in $\cX$, the Beck--Chevalley map
            \[
                \widetilde{f}_! \widetilde{g}^* \too g^* f_!
            \]
            is an isomorphism.
    \end{enumerate}
    \tdef{$\cX$-complete} $\tA$-objects are defined dually.
\end{defn}

\begin{warn}\label{warn-cocomp}
    Continuing with the case of canonically parameterized categories from \cref{can-param}, we see that our definition of $\cX$-cocompleteness is inadequate if we were to parameterize over $\cX = \Cat$ rather than $\Spaces$, as condition (2) is not satisfied.
    To see this, recall that the left Kan extension of $M \in \CC^I$ along $f\colon I \to J$ is given by
    \[
        (f_! M)_j \simeq \colim_{f(i) \to j} M_i,
    \]
    that is, the colimit of
    \[
        I_{/j} \too I \too[M] \CC.
    \]
    The Beck--Chevalley condition on the other hand states that it is isomorphic to the colimit of
    \[
        I_j \too I \too[M] \CC
    \]
    where $I_j$ is the fiber of $f$ at $j$, i.e.\ the full subcategory of $I_{/j}$ on those maps $f(i) \to j$ which are isomorphisms.
    Nevertheless, we see that if $J$ is a space then $I_j = I_{/j}$, and the issue does not arise.
\end{warn}

With this definition in mind, we also define cocontinuous morphisms in analogy to the case of parameterized categories, where we would require functors to preserve colimits, or, equivalently, commute with left Kan extensions.

\begin{defn}\label{def-cocont}
    A morphism $F\colon \CC \to \DD$ between $\cX$-cocomplete $\tA$-objects is called \tdef{$\cX$-cocontinuous} if for any $f\colon X \to Y$ in $\cX$ the Beck--Chevalley map
    \[
        f_! F_X \too F_Y f_!
    \]
    obtained from taking horizontal left adjoints in the commutative square
    % https://q.uiver.app/#q=WzAsNCxbMSwwLCJcXENDKFgpIl0sWzAsMCwiXFxDQyhZKSJdLFswLDEsIlxcREQoWSkiXSxbMSwxLCJcXEREKFgpIl0sWzEsMCwiZl4qIl0sWzIsMywiZl4qIl0sWzEsMiwiRl9ZIiwyXSxbMCwzLCJGX1giXV0=
    \[\begin{tikzcd}
        {\CC_Y} & {\CC_X} \\
        {\DD_Y} & {\DD_X}
        \arrow["{f^*}", from=1-1, to=1-2]
        \arrow["{F_Y}"', from=1-1, to=2-1]
        \arrow["{F_X}", from=1-2, to=2-2]
        \arrow["{f^*}", from=2-1, to=2-2]
    \end{tikzcd}\]
    is an isomorphism.
    \tdef{$\cX$-continuous} morphisms are defined dually.
\end{defn}

\begin{defn}
    We denote by $\mdef{\tFun^{\cocpl}(\cX^\op, \tA)} \subset \tFun(\cX^\op, \tA)$ the $2$-subcategory whose objects are the $\cX$-cocomplete $\tA$-objects and whose $1$-morphisms are the $\cX$-cocontinuous morphisms (and all higher morphisms).
    Dually, we denote by $\mdef{\tFun^{\cpl}(\cX^\op, \tA)}$ the $2$-subcategory on the $\cX$-complete $\tA$-objects and $\cX$-continuous $1$-morphisms.
    We denote by $\mdef{\tFun^{\cpl,\cocpl}(\cX^\op, \tA)}$ their intersection.
\end{defn}

When the $\tA$-objects are $\cX$-cocomplete, one can give an alternative description of $\cX$-parameterized left adjoints in terms of $\cX$-cocontinuity.

\begin{prop}\label{pla-cond}
    Let $L\colon \CC \to \DD$ be a $1$-morphism between $\cX$-cocomplete $\tA$-objects.
    Then, $L$ is a left adjoint if and only if the following two conditions holds:
    \begin{enumerate}
        \item For any $X \in \cX$, the $1$-morphism $L_X\colon \CC_X \to \DD_X$ in $\tA$ is a left adjoint.
        \item $L$ is $\cX$-cocontinuous.
    \end{enumerate}
\end{prop}

\begin{proof}
    Condition (1) is precisely condition (1) of \cref{param-adj}.
    As for condition (2), the Beck--Chevalley condition appearing in \cref{def-cocont} is equivalent to the Beck--Chevalley condition from condition (2) in \cref{param-adj} by passing to the adjoints.
\end{proof}

\subsubsection{Universal (Co)completeness}

The $2$-category of spans, which we denote by $\mdef{\Spanoh(\cX)}$, was constructed by Haugseng \cite{spans} (where it is denoted by $\Spano^+(\cX)$), to which we refer the reader for a detailed discussion.
We informally recall that this is the $2$-category whose objects are the objects of $\cX$, a $1$-morphism from $X$ to $Y$ is a span $X \gets Z \to Y$ (and composition is given by pullback), and a $2$-morphism is a commutative diagram
% https://q.uiver.app/#q=WzAsNCxbMCwxLCJYIl0sWzIsMSwiWSJdLFsxLDAsIloiXSxbMSwyLCJXIl0sWzIsM10sWzIsMF0sWzIsMV0sWzMsMF0sWzMsMV1d
\[\begin{tikzcd}[row sep=small]
	& Z \\
	X && Y \\
	& W
	\arrow[from=1-2, to=3-2]
	\arrow[from=1-2, to=2-1]
	\arrow[from=1-2, to=2-3]
	\arrow[from=3-2, to=2-1]
	\arrow[from=3-2, to=2-3]
\end{tikzcd}\]

This $2$-category enjoys a universal property, originally stated in \cite[Chapter 7, Theorem 3.2.2]{GR}, and recently proven independently by Macpherson \cite{Macph} and Stefanich \cite{Stef} (we also refer the reader to \cite[Theorem 2.2.7]{EH} for another presentation of Macpherson's proof).
We rephrase this in terms of $\cX$-cocompleteness.

\begin{thm}[{\cite[4.2.6]{Macph}, \cite[Theorem 3.4.18]{Stef}}]\label{span-1.5-univ}
    The map $\cX^\op \to \Spanoh(\cX)$ exhibits the target as the universal $2$-category with an $\cX$-cocomplete object.
    That is, pre-composition with the map $\cX^\op \to \Spanoh(\cX)$ induces an isomorphism of spaces
    \[
        \Map(\Spanoh(\cX), \tA) \iso \Map^{\cocpl}(\cX^\op, \tA).
    \]
\end{thm}

\begin{remark}\label{span-adj}
    We recall that given $f\colon X \to Y$ in $\cX$, the corresponding morphisms $f_!$ and $f^*$ in $\Spanoh(\cX)$ are the spans
    % https://q.uiver.app/#q=WzAsNixbMCwxLCJYIl0sWzIsMSwiWSJdLFsxLDAsIlgiXSxbNCwxLCJZIl0sWzYsMSwiWCJdLFs1LDAsIlgiXSxbMiwxLCJmIl0sWzIsMCwiIiwyLHsibGV2ZWwiOjIsInN0eWxlIjp7ImhlYWQiOnsibmFtZSI6Im5vbmUifX19XSxbNSwzLCJmIiwyXSxbNSw0LCIiLDAseyJsZXZlbCI6Miwic3R5bGUiOnsiaGVhZCI6eyJuYW1lIjoibm9uZSJ9fX1dXQ==
    \[\begin{tikzcd}[row sep=small]
        & X &&&& X \\
        X && Y && Y && X
        \arrow[Rightarrow, no head, from=1-2, to=2-1]
        \arrow["f", from=1-2, to=2-3]
        \arrow["f"', from=1-6, to=2-5]
        \arrow[Rightarrow, no head, from=1-6, to=2-7]
    \end{tikzcd}\]
    and the unit and counit maps exhibiting the adjunction $f_! \dashv f^*$ are the maps between spans
    % https://q.uiver.app/#q=WzAsOCxbMCwxLCJYIl0sWzEsMCwiWCJdLFsyLDEsIlgiXSxbMSwyLCJYIFxcdGltZXNfWSBYIl0sWzQsMSwiWSJdLFs2LDEsIlkiXSxbNSwwLCJYIl0sWzUsMiwiWSJdLFswLDEsIiIsMCx7ImxldmVsIjoyLCJzdHlsZSI6eyJoZWFkIjp7Im5hbWUiOiJub25lIn19fV0sWzEsMiwiIiwwLHsibGV2ZWwiOjIsInN0eWxlIjp7ImhlYWQiOnsibmFtZSI6Im5vbmUifX19XSxbMywwLCJcXHBpXzEiXSxbMywyLCJcXHBpXzIiLDJdLFsxLDMsIlxcRGVsdGEiLDFdLFs2LDQsImYiLDJdLFs2LDUsImYiXSxbNiw3LCJmIiwxXSxbNyw0LCIiLDEseyJsZXZlbCI6Miwic3R5bGUiOnsiaGVhZCI6eyJuYW1lIjoibm9uZSJ9fX1dLFs3LDUsIiIsMSx7ImxldmVsIjoyLCJzdHlsZSI6eyJoZWFkIjp7Im5hbWUiOiJub25lIn19fV1d
    \[\begin{tikzcd}[row sep=small]
        & X &&&& X \\
        X && X && Y && Y \\
        & {X \times_Y X} &&&& Y
        \arrow[Rightarrow, no head, from=1-2, to=2-3]
        \arrow["\Delta"{description}, from=1-2, to=3-2]
        \arrow["f"', from=1-6, to=2-5]
        \arrow["f", from=1-6, to=2-7]
        \arrow["f"{description}, from=1-6, to=3-6]
        \arrow[Rightarrow, no head, from=2-1, to=1-2]
        \arrow["{\pi_1}", from=3-2, to=2-1]
        \arrow["{\pi_2}"', from=3-2, to=2-3]
        \arrow[Rightarrow, no head, from=3-6, to=2-5]
        \arrow[Rightarrow, no head, from=3-6, to=2-7]
    \end{tikzcd}\]
\end{remark}

One of the advantages of having $\tA$ as a parameter (rather than restricting to $\tA = \tCat$) is that we can use this to describe not only the space of $\cX$-cocomplete $\tA$-objects, but rather their entire $2$-category.
In other words, it allows us to study $\cX$-cocontinuous morphisms.

\begin{cor}[{\cite[Corollary 2.3.5]{EH}}]\label{span-1.5-univ-fun}
    Pre-composition with the map $\cX^\op \to \Spanoh(\cX)$ induces an equivalence of $2$-categories
    \[
        \tFun(\Spanoh(\cX), \tA) \iso \tFun^{\cocpl}(\cX^\op, \tA)
    \]
\end{cor}

\begin{proof}
    We repeat the proof from \cite[Corollary 2.3.5]{EH} for the reader's convenience.
    Recall from \cite[Proposition 2.1.5]{EH} that for any $2$-category $\tB$, the exponential adjunction restricts to an isomorphism of spaces
    \[
        \Map^{\cocpl}(\cX^\op, \tFun(\tB, \tA)) \simeq \Map(\tB, \tFun^{\cocpl}(\cX^\op, \tA)).
    \]
    Using this and the universal property from \cref{span-1.5-univ} we get
    \begin{align*}
        \Map(\tB, \tFun(\Spanoh(\cX), \tA))
        &\simeq \Map(\Spanoh(\cX), \tFun(\tB, \tA))\\
        &\simeq \Map^{\cocpl}(\cX^\op, \tFun(\tB, \tA))\\
        &\simeq \Map(\tB, \tFun^{\cocpl}(\cX^\op, \tA)).
    \end{align*}
    Since this holds for any $2$-category $\tB$, the result follows from the $2$-categorical Yoneda lemma \cite[6.2.7]{Hin}.
\end{proof}

\subsubsection{Free (Co)completeness}

In \cref{span-1.5-univ} above, we have described the universal $2$-category with an $\cX$-cocomplete object.
In the case of $\tA = \tCat$, there is a related result, describing the free $\cX$-cocomplete $\cX$-parameterized category.

\begin{defn}
    We define the $\cX$-parameterized category $\mdef{\cX_{/-}}\colon \cX^\op \to \tCat$.
    Namely, the functor sending $X \in \cX$ to the over-category $\cX_{/X}$, and a map $f\colon X \to Y$ to the pullback functor $f^*\colon \cX_{/Y} \to \cX_{/X}$ sending $Z \to Y$ to the pullback of $Z \to Y \gets X$.
\end{defn}

It is easy to see that the functor $f^*$ admits a left adjoint $f_!$ given by post-composition with $f$, namely $f_!(Z \to X) = Z \to Y \to X$.
Martini--Wolf \cite[Theorem 7.1.13]{MW} showed that more is true: $\cX_{/-}$ is the free $\cX$-cocomplete $\cX$-parameterized category (see also \cite[Proposition 2.20]{CLL} for a restatement somewhat closer to our formulation, and note that one needs to take $\clB = \PSh(\cX)$ and $\clQ$ to contain only the representable morphisms, i.e.\ those in the image of the Yoneda embedding $\cX \to \PSh(\cX)$).
We state a somewhat weaker result, only concerning the value at the terminal object in $\cX$ (assuming it exists), but which avoids mentioning internal homs, and suffices for our purposes.
Here $\Fun^{\cocpl}_\cX(\CC, \DD)$ denotes the hom category in $\tFun^{\cocpl}(\cX^\op, \tCat)$.
We shall give an alternative proof based on \cref{span-1.5-univ-fun}, and refer the reader to \cite{MW} for a proof of the full statement.

\begin{prop}[{\cite[Theorem 7.1.13]{MW}}]\label{free-cocpl}
    Assume that $\cX$ has a terminal object $\pt$.
    Then, the $\cX$-parameterized category $\cX_{/-}$ is the free $\cX$-cocomplete $\cX$-parameterized category.
    Namely, for any $\cX$-cocomplete $\cX$-parameterized category $\CC\colon \cX^\op \to \tCat$, evaluation at the object $\pt \to \pt \in \cX_{/\pt}$ induces an equivalence
    \[
        \Fun^{\cocpl}_\cX(\cX_{/-}, \CC) \iso \CC_\pt
        \qin \tCat.
    \]
\end{prop}

\begin{proof}
    Recall that the hom category in $\Spanoh(\cX)$ is given by the over-category
    \[ 
        \Hom_{\Spanoh(\cX)}(X, Y) \simeq \cX_{/X \times Y},
    \]
    for example by \cite[Proposition 8.3]{spans} (by taking $k=2$ and restricting to the $2$-subcategory whose $2$-morphisms are spans between spans where the wrong-way map is the identity), or \cite[Theorem 2.3.22]{EH}, or \cite[Remark 3.1.10]{Stef}.
    Taking the source to be $\pt \in \Spanoh(\cX)$, we obtain a natural equivalence
    \[
        \cX_{/-} \simeq \Hom_{\Spanoh(\cX)}(\pt, -)\colon \Spanoh(\cX) \too \tCat.
    \]
    In particular, we see that $\cX_{/-}$ lifts to a $2$-functor from $\Spanoh(\cX)$, hence, by the universal property from \cref{span-1.5-univ}, it is $\cX$-cocomplete.

    Since $\CC$ is $\cX$-cocomplete, it lifts to $\Spanoh(\cX)$ too by \cref{span-1.5-univ}.
    Using \cref{span-1.5-univ-fun}, the equivalence above, and the $2$-categorical Yoneda lemma \cite[6.2.7]{Hin}, we conclude that
    \begin{align*}
        \Fun^{\cocpl}_\cX(\cX_{/-}, \CC)
        &\simeq \Fun_{\Spanoh(\cX)}(\cX_{/-}, \CC)\\
        &\simeq \Fun_{\Spanoh(\cX)}(\Hom_{\Spanoh(\cX)}(\pt, -), \CC)\\
        &\simeq \CC_\pt.
    \end{align*}
\end{proof}

\subsection{Parameterized Semiadditivity}

\subsubsection{Definitions and Basic Properties}

Parameterized semiadditivity is originally due to \cite[Construction 4.1.8]{HL}, and we refer the reader to \cite[\S3]{CLL} for a recent discussion closer to our perspective.
These sources consider the ambidexterity only of truncated morphisms, and define it inductively on the truncation level.
Nevertheless, there is a non-inductive definition, which also applies to all morphisms, and appears for instance in \cite[Definition 3.5]{TwA} (though in the symmetric monoidal context).
In our applications we shall only consider the truncated case, but we find the non-inductive definition more natural and easier to handle.
We also generalize it to the case of an arbitrary $2$-category in place of $\tCat$.

\begin{defn}\label{Nm-def}
    Let $\CC\colon \cX^\op \to \tA$ be $\cX$-complete and $\cX$-cocomplete, let $f\colon X \to Y$, and consider the commutative diagram
    % https://q.uiver.app/#q=WzAsNSxbMSwxLCJYIFxcdGltZXNfWSBYIl0sWzIsMSwiWCJdLFsxLDIsIlgiXSxbMiwyLCJZIl0sWzAsMCwiWCJdLFswLDIsIlxccGlfMSIsMl0sWzAsMSwiXFxwaV8yIl0sWzEsMywiZiJdLFsyLDMsImYiLDJdLFs0LDIsIiIsMix7ImN1cnZlIjozLCJsZXZlbCI6Miwic3R5bGUiOnsiaGVhZCI6eyJuYW1lIjoibm9uZSJ9fX1dLFs0LDEsIiIsMix7ImN1cnZlIjotMywibGV2ZWwiOjIsInN0eWxlIjp7ImhlYWQiOnsibmFtZSI6Im5vbmUifX19XSxbNCwwLCJcXERlbHRhIl0sWzAsMywiIiwwLHsic3R5bGUiOnsibmFtZSI6ImNvcm5lciJ9fV1d
    \[\begin{tikzcd}
        X \\
        & {X \times_Y X} & X \\
        & X & Y
        \arrow["\Delta", from=1-1, to=2-2]
        \arrow[curve={height=-18pt}, Rightarrow, no head, from=1-1, to=2-3]
        \arrow[curve={height=18pt}, Rightarrow, no head, from=1-1, to=3-2]
        \arrow["{\pi_2}", from=2-2, to=2-3]
        \arrow["{\pi_1}"', from=2-2, to=3-2]
        \arrow["\lrcorner"{anchor=center, pos=0.125}, draw=none, from=2-2, to=3-3]
        \arrow["f", from=2-3, to=3-3]
        \arrow["f"', from=3-2, to=3-3]
    \end{tikzcd}\]
    We define the \tdef{dualizing map} to be $\mdef{D_f} := \pi_{1*} \Delta_!\colon \CC_X \to \CC_X$.
    We define the (twisted) \tdef{norm map} $\mdef{\Nm_f}\colon f_! D_f \to f_*$ to be the composition
    \[
        f_! D_f
        := f_! \pi_{1*} \Delta_!
        \too[u^f_*] f_* f^* f_! \pi_{1*} \Delta_!
        \too[\beta_!^{-1}] f_* \pi_{2!} \pi_1^* \pi_{1*} \Delta_!
        \too[c^{\pi_1}_*] f_* \pi_{2!} \Delta_!
        \simeq f_*,
    \]
    where $u^f_*$ is the unit of the adjunction $f^* \dashv f_*$, $\beta_!$ is the Beck--Chevalley map corresponding to square in the diagram above, and $c^{\pi_1}_*$ is the counit of the adjunction $\pi_1^* \dashv \pi_{1*}$.
    We say that $f$ is (twisted) \tdef{$\CC$-ambidexterous} if $\Nm_f$ is an isomorphism.
\end{defn}

We connect this definition to the inductive definition.

\begin{prop}\label{ambi-defs}
    Let $\CC\colon \cX^\op \to \tA$ be $\cX$-complete and $\cX$-cocomplete, and let $m \geq -2$.
    If all $(m-1)$-truncated morphisms are $\CC$-ambidexterous, then for any $m$-truncated morphism $f\colon X \to Y$, there is a canonical isomorphism $D_f \simeq \Id_{\CC_X}$, and in particular the map $\Nm_f$ has source canonically isomorphic to $f_!$.
\end{prop}

\begin{proof}
    We prove this inductively.
    The case $m=-2$ is trivial, as $f$, $\pi_1$, $\pi_2$ and $\Delta$ are the identity maps.
    For $m > -2$, let $f\colon X \to Y$ be an $m$-truncated morphism in $\cX$.
    By definition, the diagonal $\Delta$ is $(m-1)$-truncated, thus, by the inductive hypothesis, the norm map induces an isomorphism $\Nm_{\Delta}\colon \Delta_! \iso \Delta_*$.
    We conclude that
    \[
        D_f = \pi_{1*} \Delta_! \iso \pi_{1*} \Delta_* \simeq \Id_{\CC_X},
    \]
    as required.
\end{proof}

With the definition of $\CC$-ambidexterity in hand, we can now define $\cX$-semiadditivity.

\begin{defn}
    We say that an $\cX$-complete and $\cX$-cocomplete $\CC\colon \cX^\op \to \tA$ is \tdef{$\cX$-semiadditive} if all $f\colon X \to Y$ are $\CC$-ambidexterous, i.e.\ the associated norm map $\Nm_f\colon f_! D_f \to f_*$ is an isomorphism.

    We say that an $\cX$-parameterized map between $\cX$-semiadditive $\tA$-objects $F\colon \CC \to \DD$ is \tdef{$\cX$-semiadditive} if it is $\cX$-cocontinuous and $\cX$-continuous.

    We thus denote by $\tFun^{\psa}(\cX^\op, \tA) \subset \tFun^{\cpl,\cocpl}(\cX^\op, \tA)$ the $2$-subcategory on the $\cX$-semiadditive objects and all morphisms.
\end{defn}

\cref{ambi-defs} shows that our definition of $\cX$-semiadditivity is a direct generalization of the ordinary definition of higher semiadditivity of \cite{HL} from the truncated case.

\begin{example}\label{can-param-sa}
    Let $\cX = \Spacespi$ be the collection of $\pi$-finite spaces and $\tA = \tCat$.
    Then, for a (non-parameterized) category $\CC$, the canonically parameterized category $\CC\colon \Spacespi^\op \to \tCat$ of \cref{can-param} is $\Spacespi$-parameterized $\Spacespi$-semiadditive if and only if $\CC$ is $\infty$-semiadditive.
\end{example}

Note that our definition of $\cX$-semiadditive \emph{morphisms} is simply about being $\cX$-continuous and $\cX$-cocontinuous, and does not involve the norm maps, similarly to the non-parameterized case.
It not a priori clear that we have defined the ``correct'' ($1$- or higher) morphisms in the $2$-category of $\cX$-semiadditive objects.
We shall now prove \cref{inm-exp-adj}, stating that it is indeed the case.

Recall from \cite[Proposition 2.1.5]{EH} that for a category $\cX$ and $2$-categories $\tA$ and $\tB$, the exponential adjunction
\[
    \Map(\cX^\op, \tFun(\tB, \tA)) \simeq \Map(\tB, \tFun(\cX^\op, \tA))
\]
restricts to an isomorphism of spaces on $\cX$-complete and $\cX$-cocomplete objects
\[
    \Map^{\cpl,\cocpl}(\cX^\op, \tFun(\tB, \tA)) \simeq \Map(\tB, \tFun^{\cpl,\cocpl}(\cX^\op, \tA)),
\]
showing that $\tFun^{\cpl,\cocpl}(\cX^\op, \tA)$ captures the ``correct'' $2$-category of $\cX$-complete and $\cX$-cocomplete $\tA$-objects and (higher) morphisms between them.
Using this we deduce the corresponding statement for $\cX$-semiadditive objects.

\begin{lem}\label{inm-exp-adj}
    The exponential adjunction restricts to an isomorphism of spaces
    \[
        \Map^{\psa}(\cX^\op, \tFun(\tB, \tA)) \simeq \Map(\tB, \tFun^{\psa}(\cX^\op, \tA)).
    \]
\end{lem}

\begin{proof}
    As mentioned above, the corresponding claim for $\cX$-complete and $\cX$-cocomplete objects holds by \cite[Proposition 2.1.5]{EH}.
    It remains to show that it restricts to $\cX$-semiadditive objects.
    Namely, given an $\cX$-complete and $\cX$-cocomplete $\tFun(\tB, \tA)$-object $\CC\colon \cX^\op \to \tFun(\tB, \tA)$, we need to show that it is $\cX$-semiadditive if and only if for every $b \in \tB$ the $\cX$-parameterized $\tA$-object $\CC_b\colon \cX^\op \to \tA$ is $\cX$-semiadditive.
    To that end, it suffices to show that for every $f\colon X \to Y$ in $\cX$, the norm map associated to $\CC$ is an isomorphism if and only if the norm map associated to $\CC_b$ is an isomorphism for every $b \in \tB$.
    Since evaluation at $b$ is a $2$-functor $\tFun(\tB, \tA) \to \tA$, it sends the norm map associated to $\CC$ to the norm map associated to $\CC_b$, and the claim now follows from \cite[Lemma B.2.9]{MGS}.
\end{proof}

The next observation is that if $\cX$ is truncated, then a morphism between $\cX$-semiadditive $\tA$-objects is $\cX$-cocontinuous if and only if it is $\cX$-continuous.

\begin{prop}\label{cocont-sa}
    Assume that $\cX$ is truncated.
    Let $\CC$ and $\DD$ be $\cX$-semiadditive, and let $F\colon \CC \to \DD$ be a map between them.
    If $F$ is $\cX$-cocontinuous, then $F$ is $\cX$-semiadditive.
\end{prop}

\begin{proof}
    The proof for the case $\tA = \tCat$ from \cite[Corollary 3.17]{CLL} works in verbatim for general $\tA$ (and is itself identical to the non-parameterized case from \cite[Corollary 3.2.4]{TeleAmbi}).
\end{proof}

In particular, we see that an $\cX$-parameterized left adjoint is automatically $\cX$-semiadditive.

\begin{prop}\label{left-sa}
    Assume that $\cX$ is truncated.
    Let $\CC$ and $\DD$ be $\cX$-semiadditive, and let $L\colon \CC \to \DD$ be a map between them.
    If $L$ is $\cX$-parameterized left adjoint, then it is $\cX$-semiadditive.
\end{prop}

\begin{proof}
    By \cref{pla-cond}, an $\cX$-parameterized left adjoint is $\cX$-cocontinuous, thus by \cref{cocont-sa} it is $\cX$-semiadditive.
\end{proof}

\subsubsection{Universal Semiadditivity}

Recall that in \cref{span-1.5-univ} we have describe the universal $2$-category with an $\cX$-cocomplete object.
We now concern ourselves with the universal $2$-category with an $\cX$-semiadditive object.
It is in fact easy to show that there is such a universal $2$-category, while the difficulty lies in describing it explicitly, a situation which is completely analogues to the case of $\cX$-cocompleteness and $\Spanoh(\cX)$ as is explained in \cite[\S2.3]{EH}.

We begin by constructing the universal $2$-category with an $\cX$-semiadditive object (note that the following does not require us to assume that $\cX$ is truncated).

\begin{prop}\label{usa-univ}
    There exists a $2$-category $\mdef{\Usa(\cX)}$ and a map $\cX^\op \to \Usa(\cX)$ exhibiting it as the universal $2$-category with an $\cX$-semiadditive object.
    That is, pre-composition with the map $\cX^\op \to \Usa(\cX)$ induces an isomorphism of spaces
    \[
        \Map(\Usa(\cX), \tA) \iso \Map^{\psa}(\cX^\op, \tA).
    \]
\end{prop}

\begin{proof}
    First, consider the pushout
    % https://q.uiver.app/#q=WzAsNCxbMCwwLCJcXGNYXlxcb3AiXSxbMSwwLCJcXFNwYW5vaChcXGNYKSJdLFswLDEsIlxcU3Bhbm9oKFxcY1gpXntcXGNvfSJdLFsxLDEsIlxcVWNwbGNvY3BsKFxcY1gpIl0sWzIsM10sWzEsM10sWzAsMl0sWzAsMV0sWzMsMCwiIiwxLHsic3R5bGUiOnsibmFtZSI6ImNvcm5lciJ9fV1d
    \[\begin{tikzcd}
        {\cX^\op} & {\Spanoh(\cX)} \\
        {\Spanoh(\cX)^{\co}} & {\Ucplcocpl(\cX)}
        \arrow[from=1-1, to=1-2]
        \arrow[from=1-1, to=2-1]
        \arrow[from=1-2, to=2-2]
        \arrow[from=2-1, to=2-2]
        \arrow["\lrcorner"{anchor=center, pos=0.125, rotate=180}, draw=none, from=2-2, to=1-1]
    \end{tikzcd}\]
    By \cref{span-1.5-univ}, the map $\cX^\op \to \Spanoh(\cX)$ exhibits the target as the universal $2$-category with an $\cX$-cocomplete object.
    Dually, $\cX^\op \to \Spanoh(\cX)^{\co}$ exhibits the target as the universal $2$-category with an $\cX$-complete object.
    Therefore, the canonical map $\cX^\op \to \Ucplcocpl(\cX)$ exhibits the target as the universal $2$-category with an $\cX$-complete and $\cX$-cocomplete object, that is, it induces an isomorphism
    \[
        \Map(\Ucplcocpl(\cX), \tA) \iso \Map^{\cpl,\cocpl}(\cX^\op, \tA).
    \]

    In \cref{Nm-def} we have defined the norm maps associated to an $\cX$-complete and $\cX$-cocomplete object.
    Applying this to $\cX^\op \to \Ucplcocpl(\cX)$, we get a collection of $2$-morphisms $N := \{\Nm_f\}$ in $\Ucplcocpl(\cX)$, ranging over all $1$-morphisms in $\cX$.
    Then, the localization $\Usa(\cX) := \Ucplcocpl(\cX)[N^{-1}]$, equipped with the canonical map $\cX^\op \to \Ucplcocpl(\cX) \to \Usa(\cX)$, is by definition the universal $2$-category with an $\cX$-semiadditive object.
\end{proof}

Following the exact same strategy as in the case of $\Spanoh(\cX)$ from \cref{span-1.5-univ-fun}, varying the $2$-category $\tA$ allows us to describe the entire $2$-category of $\cX$-semiadditive objects and $\cX$-semiadditive morphisms between them.

\begin{cor}\label{usa-univ-fun}
    Pre-composition with the map $\cX^\op \to \Usa(\cX)$ induces an equivalence of $2$-categories
    \[
        \tFun(\Usa(\cX), \tA) \iso \tFun^{\psa}(\cX^\op, \tA).
    \]
\end{cor}

\begin{proof}
    Recall from \cref{inm-exp-adj} that for any $2$-category $\tB$, the exponential adjunction restricts to an isomorphism of spaces
    \[
        \Map^{\psa}(\cX^\op, \tFun(\tB, \tA)) \simeq \Map(\tB, \tFun^{\psa}(\cX^\op, \tA)).
    \]
    Using this and the universal property from \cref{usa-univ} we get
    \begin{align*}
        \Map(\tB, \tFun(\Usa(\cX), \tA))
        &\simeq \Map(\Usa(\cX), \tFun(\tB, \tA))\\
        &\simeq \Map^{\psa}(\cX^\op, \tFun(\tB, \tA))\\
        &\simeq \Map(\tB, \tFun^{\psa}(\cX^\op, \tA)).
    \end{align*}
    Since this holds for any $2$-category $\tB$, the result follows from the $2$-categorical Yoneda lemma \cite[6.2.7]{Hin}.
\end{proof}

Consider now the $2$-category of iterated spans $\mdef{\Spant(\cX)}$, constructed by Haugseng \cite{spans}.
Informally, it has the same objects as $\cX$, a $1$-morphism from $X$ to $Y$ is a span $X \gets Z \to Y$ (and composition is given by pullback), and a $2$-morphism is a span between spans, i.e.\ a commutative diagram
% https://q.uiver.app/#q=WzAsNSxbMCwxLCJYIl0sWzIsMSwiWSJdLFsxLDAsIloiXSxbMSwyLCJXIl0sWzEsMSwiUSJdLFsyLDBdLFsyLDFdLFszLDBdLFszLDFdLFs0LDJdLFs0LDNdXQ==
\[\begin{tikzcd}[row sep=small]
	& Z \\
	X & Q & Y \\
	& W
	\arrow[from=1-2, to=2-1]
	\arrow[from=1-2, to=2-3]
	\arrow[from=3-2, to=2-1]
	\arrow[from=3-2, to=2-3]
	\arrow[from=2-2, to=1-2]
	\arrow[from=2-2, to=3-2]
\end{tikzcd}\]

We recall that there is a canonical factorization
\[
    \cX^\op \too \Spanoh(\cX) \too \Spant(\cX),
\]
which is given on objects and $1$-morphisms by the obvious maps, and on $2$-morphisms by
% https://q.uiver.app/#q=WzAsOSxbMCwxLCJYIl0sWzIsMSwiWSJdLFsxLDAsIloiXSxbMSwyLCJXIl0sWzQsMSwiWCJdLFs1LDAsIloiXSxbNSwyLCJXIl0sWzUsMSwiWiJdLFs2LDEsIlkiXSxbMiwwXSxbMiwxXSxbMywwXSxbMywxXSxbMiwzXSxbNSw3LCIiLDEseyJsZXZlbCI6Miwic3R5bGUiOnsiaGVhZCI6eyJuYW1lIjoibm9uZSJ9fX1dLFs3LDZdLFs1LDRdLFs1LDhdLFs2LDRdLFs2LDhdLFsxLDQsIiIsMSx7InNob3J0ZW4iOnsic291cmNlIjo0MCwidGFyZ2V0Ijo0MH0sInN0eWxlIjp7InRhaWwiOnsibmFtZSI6Im1hcHMgdG8ifX19XV0=
\[\begin{tikzcd}[row sep=small]
	& Z &&&& Z \\
	X && Y && X & Z & Y \\
	& W &&&& W
	\arrow[from=1-2, to=2-1]
	\arrow[from=1-2, to=2-3]
	\arrow[from=1-2, to=3-2]
	\arrow[from=1-6, to=2-5]
	\arrow[Rightarrow, no head, from=1-6, to=2-6]
	\arrow[from=1-6, to=2-7]
	\arrow[shorten <=19pt, shorten >=19pt, maps to, from=2-3, to=2-5]
	\arrow[from=2-6, to=3-6]
	\arrow[from=3-2, to=2-1]
	\arrow[from=3-2, to=2-3]
	\arrow[from=3-6, to=2-5]
	\arrow[from=3-6, to=2-7]
\end{tikzcd}\]
Observe that this factorization in particular exhibits $\Spant(\cX)$ as $\cX$-cocomplete by \cref{span-1.5-univ}.
Dually, there is a canonical factorization
\[
    \cX^\op \too \Spanoh(\cX)^{\co} \too \Spant(\cX),
\]
exhibiting $\Spant(\cX)$ as $\cX$-complete.

Recall from \cref{span-adj} that given $f\colon X \to Y$, there is an explicit presentation of $f_!$ and $f^*$ as well as their adjunction data in $\Spanoh(\cX)$.
In line with this, we can present both morphisms $f_!$ and $f_*$ in $\Spant(\cX)$ as the same span, picking an identification $f_! = f_*$, with corresponding adjunction data.
We now prove that not only are $f_!$ and $f_*$ identified, but through this identification, the norm maps and dualizing maps are canonically isomorphic to the identity.
In particular, this shows that the map $\cX^\op \to \Spant(\cX)$ exhibits an $\cX$-semiadditive object.
We thank Lior Yanovski for suggesting a simplified approach to the proof.

\begin{prop}\label{span-2-sa}
    The map $\cX^\op \to \Spant(\cX)$ exhibits an $\cX$-semiadditive object.
    Moreover, under the presentation above, for any map $f\colon X \to Y$ in $\cX$, the $1$-morphism $D_f$ is canonically isomorphic to $\Id_X$, and the $2$-morphism $\Nm_f$ is canonically isomorphic to $\Id_{\Id_{f_*}}$.
\end{prop}

\begin{proof}
    We have already explained above that $\Spant(\cX)$ is $\cX$-complete and $\cX$-cocomplete, so it remains to show that for any map $f\colon X \to Y$ in $\cX$, the norm map and the dualizing map are both canonically isomorphic to the identity.
    For the convenience of the reader we recall their construction from \cref{Nm-def}.
    Consider the commutative diagram
    % https://q.uiver.app/#q=WzAsNSxbMSwxLCJYIFxcdGltZXNfWSBYIl0sWzIsMSwiWCJdLFsxLDIsIlgiXSxbMiwyLCJZIl0sWzAsMCwiWCJdLFswLDIsIlxccGlfMSIsMl0sWzAsMSwiXFxwaV8yIl0sWzEsMywiZiJdLFsyLDMsImYiLDJdLFs0LDIsIiIsMix7ImN1cnZlIjozLCJsZXZlbCI6Miwic3R5bGUiOnsiaGVhZCI6eyJuYW1lIjoibm9uZSJ9fX1dLFs0LDEsIiIsMix7ImN1cnZlIjotMywibGV2ZWwiOjIsInN0eWxlIjp7ImhlYWQiOnsibmFtZSI6Im5vbmUifX19XSxbNCwwLCJcXERlbHRhIl0sWzAsMywiIiwwLHsic3R5bGUiOnsibmFtZSI6ImNvcm5lciJ9fV1d
    \[\begin{tikzcd}
        X \\
        & {X \times_Y X} & X \\
        & X & Y
        \arrow["\Delta", from=1-1, to=2-2]
        \arrow[curve={height=-18pt}, Rightarrow, no head, from=1-1, to=2-3]
        \arrow[curve={height=18pt}, Rightarrow, no head, from=1-1, to=3-2]
        \arrow["{\pi_2}", from=2-2, to=2-3]
        \arrow["{\pi_1}"', from=2-2, to=3-2]
        \arrow["\lrcorner"{anchor=center, pos=0.125}, draw=none, from=2-2, to=3-3]
        \arrow["f", from=2-3, to=3-3]
        \arrow["f"', from=3-2, to=3-3]
    \end{tikzcd}\]
    We let $D_f := \pi_{1*} \Delta_!$, and define $\Nm_f\colon f_! D_f \to f_*$ to be the composition
    \[
        f_! D_f
        = f_! \pi_{1*} \Delta_!
        \too[u^f_*] f_* f^* f_! \pi_{1*} \Delta_!
        \too[\beta_!^{-1}] f_* \pi_{2!} \pi_1^* \pi_{1*} \Delta_!
        \too[c^{\pi_1}_*] f_* \pi_{2!} \Delta_!
        \simeq f_*.
    \]

    For the dualizing map, note that $\pi_1 \Delta \simeq \Id_X$, so that under our identification, we get
    \[
        D_f := \pi_{1*} \Delta_!
        \simeq \pi_{1*} \Delta_*
        \simeq \Id_X
    \]
    as required.

    We now move on to the norm map.
    First, recall that composition in $\Spant(\cX)$ is given by forming pullbacks, thus, as the square in the diagram above is a pullback square, the Beck--Chevalley map
    \[
        \beta_!\colon \pi_{2!} \pi_1^* \too f^* f_!
    \]
    is identified with the identity $2$-morphism.
    Similarly, the same holds for the Beck--Chevalley map
    \[
        \beta_*\colon f^* f_* \too \pi_{2*} \pi_1^*.
    \]
    In particular, $\beta_!^{-1}$ is identified with $\beta_*$, so that the norm map is identified with the composition
    \[
        f_*
        \simeq f_* \pi_{1*} \Delta_*
        \too[u^f_*] f_* f^* f_* \pi_{1*} \Delta_*
        \too[\beta_*] f_* \pi_{2*} \pi_1^* \pi_{1*} \Delta_*
        \too[c^{\pi_1}_*] f_* \pi_{2*} \Delta_*
        \simeq f_*.
    \]
    Finally, recall that the Beck--Chevalley map $\beta_*$ is given by the composition
    \[
        f^* f_* \too[u^{\pi_1}_*] f^* f_* \pi_{1*} \pi_1^* \iso f^* f_* \pi_{2*} \pi_1^* \too[c^f_*] \pi_{2*} \pi_1^*,
    \]
    so that the norm is the identity by the zigzag identities for the adjunctions $f^* \dashv f_*$ and $\pi_1^* \dashv \pi_{1*}$.
\end{proof}

As an immediate consequence, since $\Usa(\cX)$ is the universal $2$-category with an $\cX$-semiadditive object (\cref{usa-univ}), we get a canonical factorization
\[
    \cX^\op \too \Usa(\cX) \too \Spant(\cX).
\]
We note that \cref{span-2-sa} does not state that $\Spant(\cX)$ is the universal $2$-category with an $\cX$-semiadditive object.
In fact, this can not be true for any $\cX$, as for example the dualizing map may not be invertible.
Having said that, by \cref{ambi-defs}, when $\cX$ is truncated we know that $D_f$ is the identity, suggesting that in this case $\Spant(\cX)$ is the universal $2$-category with an $\cX$-semiadditive object, namely that the map $\Usa(\cX) \to \Spant(\cX)$ above is an equivalence.

\begin{conj}\label{span-2-univ}
    Assume that $\cX$ is truncated.
    The map $\cX^\op \to \Spant(\cX)$ exhibits the target as the universal $2$-category with an $\cX$-semiadditive object.
    That is, it induces an equivalence of $2$-categories
    \[
        \tFun(\Spant(\cX), \tA) \iso \tFun^{\psa}(\cX^\op, \tA).
    \]
\end{conj}

\begin{remark}
    This conjecture was loosely described by Hopkins--Lurie in \cite[Remark 4.2.5]{HL}.
    Furthermore, Balmer--Dell'Ambrogio \cite[6.1.13]{BD} have proven a $(2,2)$-categorical version of the conjecture, supporting its validity.
    Finally, as we explain below in \cref{free-sa-univ-sa}, the conjecture is also supported by the description of the free $\cX$-semiadditive $\cX$-parameterized category due to Cnossen--Lenz--Linskens \cite{CLL}.
\end{remark}

\subsubsection{Free Semiadditivity}

Recall that \cref{free-cocpl} describes the free $\cX$-cocomplete $\cX$-parameterized category, given by $\cX_{/-}\colon \cX^\op \to \tCat$.
Analogously, Cnossen--Lenz--Linskens \cite{CLL} recently described the free $\cX$-semiadditive $\cX$-parameterized category, extending the result of Harpaz from the canonically $\Spacespi$-parameterized case \cite{Harpaz}.

\begin{defn}[{\cite[Construction 4.1]{CLL}}]
    We define the $\cX$-parameterized category
    \[
        \mdef{\Spano(\cX_{/-})}\colon \cX^\op \too \tCat
    \]
    by applying $\Spano$ level-wise.
    Namely, this is the functor sending $X \in \cX$ to the over-category $\Spano(\cX_{/X})$, and a map $f\colon X \to Y$ to the pullback functor $f^*\colon \Spano(\cX_{/X}) \to \Spano(\cX_{/Y})$.
\end{defn}

We now state the main result of Cnossen--Lenz--Linskens \cite{CLL}, albeit in a somewhat weaker form, only concerning the value at the terminal object in $\cX$ (assuming it exists), but once more avoiding mentioning internal homs.
Here $\Fun^{\psa}_\cX(\CC, \DD)$ denotes the hom category in $\tFun^{\psa}(\cX^\op, \tCat)$.

\begin{prop}[{\cite[Theorem A]{CLL}}]\label{free-sa}
    Assume that $\cX$ is truncated and has a terminal object $\pt$.
    Then, the $\cX$-parameterized category $\Spano(\cX_{/-})$ is the free $\cX$-semiadditive $\cX$-parameterized category.
    Namely, for any $\cX$-semiadditive $\cX$-parameterized category $\CC\colon \cX^\op \to \tCat$, evaluation at the object $\pt \to \pt \in \Spano(\cX_{/\pt})$ induces an equivalence
    \[
        \Fun^{\psa}_\cX(\Spano(\cX_{/-}), \CC) \iso \CC_\pt
        \qin \tCat.
    \]
\end{prop}

\begin{remark}\label{free-sa-univ-sa}
    As we have seen previously, Martini--Wolf's description of the free $\cX$-cocomplete $\cX$-parameterized category from \cref{free-cocpl} follows from the universal property of the $2$-category of spans described in \cref{span-1.5-univ}.
    There is a similar connection between the conjectural universal property of the $2$-category of iterated spans from \cref{span-2-univ} and \cref{free-sa}, as we have
    \[
        \Hom_{\Spant(\cX)}(\pt, X) \simeq \Spano(\cX_{/X}),
    \]
    providing further evidence for the conjecture.
\end{remark}

	\section{Parameterized Decategorification}\label{sec-decat}

\begin{notn}
    Throughout this section, $\mdef{\cX}$ is a category with pullbacks and a terminal object $\pt \in \cX$.
    For $X \in \cX$, we denote by $\mdef{f_X}\colon X \to \pt$ the unique map.
\end{notn}

Let $\CC\colon \cX^\op \to \tCat$ be an $\cX$-parameterized category.
If $\CC$ is $\cX$-cocomplete, then for any $X \in \cX$ and $M \in \CC_\pt$, we can decategorify, and define the \emph{$\cX$-parameterized homology} to be
\[
    M_\CC[X] := f_{X!} f_X^* M
    \qin \CC_\pt.
\]
Dually, if $\CC$ is $\cX$-complete, we define the \emph{$\cX$-parameterized cohomology} to be
\[
    M_\CC^X := f_{X*} f_X^* M
    \qin \CC_\pt.
\]

\begin{example}\label{hom-ord}
    In the case where $\cX = \Spaces$ and $\CC = \Sp$ is the canonically $\Spaces$-parameterized category associated to the category of spectra (see \cref{can-param}), the $\Spaces$-parameterized homology is given by $M_\Sp[X] \simeq \colim_X M$, i.e.\ it agrees with the usual homology functor, and dually for cohomology.
\end{example}

Recall that if $\cX$ is truncated, and $\CC\colon \cX^\op \to \tCat$ is $\cX$-semiadditive, the norm map provides an isomorphism $\Nm_f\colon f_! \iso f_*$.
In particular, this identifies $\cX$-parameterized homology and cohomology
\[
    M_\CC[X] \simeq M_\CC^X
    \qin \CC_\pt.
\]
As such, the functoriality of $\cX$-parameterized homology in $X$ endows the $\cX$-parameterized cohomology with a wrong-way functoriality, which we call \emph{integration} along $f$
\[
    \int_f\colon M_\CC^X \simeq M_\CC[X] \too M_\CC[Y] \simeq M_\CC^Y.
\]

This in turn allows for a further decategorification step.
We define the \emph{$\cX$-parameterized cardinality} of $X$ at $M$ to be the composition of the right- and wrong-way maps along $f_X\colon X \to \pt$ on $\cX$-parameterized cohomology
\[
    |X|_\CC\colon M \too M_\CC^X \too[\int_X] M.
\]

The purpose of this section is to make these constructions fully coherent in all parameters, making use of the free $\cX$-cocomplete and free $\cX$-semiadditive $\cX$-parameterized categories from \cref{free-cocpl} and \cref{free-sa}.

\subsection{Parameterized (Co)homology}

\begin{defn}\label{param-hom}
    Let $\CC\colon \cX^\op \to \tCat$ be an $\cX$-cocomplete $\cX$-parameterized category.
    We define the \tdef{$\cX$-parameterized homology} functor to be the composition
    \[
        \CC_\pt
        \iso \Fun^{\cocpl}_\cX(\cX_{/-}, \CC)
        \too \Fun(\cX, \CC_\pt),
        \qquad M \mapsto \mdef{M_\CC[-]}
    \]
    of the inverse of the isomorphism from \cref{free-cocpl} with the functor restricting an $\cX$-cocontinuous $\cX$-parameterized functor to its value at $\pt \in \cX$.
\end{defn}

We can present the construction in a different way.
Recall that we gave a proof of \cref{free-cocpl} relying on the universal property of the $2$-category of spans from \cref{span-1.5-univ}.
This universal property implies that $\CC$ lifts to a $2$-functor
\[
    \CC\colon \Spanoh(\cX) \too \tCat,
\]
which induces a functor from the endomorphism category of $\pt \in \Spanoh(\cX)$ to the endomorphism category of $\CC_\pt \in \tCat$
\[
    \cX \simeq \End_{\Spanoh(\cX)}(\pt) \too \End(\CC_\pt).
\]
It follows from the construction that this is the mate of the functor from \cref{param-hom} under the exponential adjunction.
This perspective allows to easily verify that $M_\CC[-]$ takes the values we expect on objects and morphisms of $\cX$.

\begin{prop}\label{hom-counit}
    Let $\CC\colon \cX^\op \to \tCat$ be an $\cX$-cocomplete $\cX$-parameterized category.
    For any $X \in \cX$ we have an isomorphism
    \[
        M_\CC[X] \simeq f_{X!} f_X^* M
        \qin \CC_\pt
    \]
    naturally in $M \in \CC_\pt$.
    For any map $f\colon X \to Y$ the induced morphism $M_\CC[X] \to M_\CC[Y]$ is isomorphic to the composition
    \[
        M_\CC[X] \simeq f_{X!} f_X^* M \simeq f_{Y!} f_! f^* f_Y^* \too[c^f_!] f_{Y!} f_Y^* \simeq M_\CC[Y]
    \]
    where $c^f_!$ is the counit of the adjunction $f_! \dashv f^*$, naturally in $M \in \CC_\pt$.
\end{prop}

\begin{proof}
    Recall that under the equivalence $\cX \simeq \End_{\Spanoh(\cX)}(\pt)$, the object $X \in \cX$ corresponds to the span
    \[
        \pt \oot[f_X] X \too[f_X] \pt.
    \]
    Thus, $M_\CC[X]$ is the image of this span under the lift $\CC\colon \Spanoh(\cX) \to \tCat$, proving the first part.
    
    For the second part, we are interested in the image of the following $2$-morphism
    % https://q.uiver.app/#q=WzAsNCxbMCwxLCJcXHB0Il0sWzEsMCwiWCJdLFsyLDEsIlxccHQiXSxbMSwyLCJZIl0sWzEsMCwiZl9YIiwyXSxbMSwyLCJmX1giXSxbMywwLCJmX1kiXSxbMywyLCJmX1kiLDJdLFsxLDMsImYiLDFdXQ==
    \[\begin{tikzcd}[row sep=small]
        & X \\
        \pt && \pt \\
        & Y
        \arrow["{f_X}"', from=1-2, to=2-1]
        \arrow["{f_X}", from=1-2, to=2-3]
        \arrow["f"{description}, from=1-2, to=3-2]
        \arrow["{f_Y}", from=3-2, to=2-1]
        \arrow["{f_Y}"', from=3-2, to=2-3]
    \end{tikzcd}\]
    Observe that this $2$-morphism is equivalent to the following composition of three $2$-morphisms
    % https://q.uiver.app/#q=WzAsMTAsWzMsMCwiWCJdLFszLDIsIlkiXSxbMiwxLCJZIl0sWzQsMSwiWSJdLFswLDEsIlxccHQiXSxbMSwyLCJZIl0sWzUsMCwiWSJdLFs2LDEsIlxccHQiXSxbMSwwLCJZIl0sWzUsMiwiWSJdLFswLDEsImYiLDFdLFswLDIsImYiLDJdLFswLDMsImYiXSxbMiwxLCIiLDIseyJsZXZlbCI6Miwic3R5bGUiOnsiaGVhZCI6eyJuYW1lIjoibm9uZSJ9fX1dLFsxLDMsIiIsMix7ImxldmVsIjoyLCJzdHlsZSI6eyJoZWFkIjp7Im5hbWUiOiJub25lIn19fV0sWzUsMiwiIiwyLHsibGV2ZWwiOjIsInN0eWxlIjp7ImhlYWQiOnsibmFtZSI6Im5vbmUifX19XSxbNSw0LCJmX1kiXSxbNiw3LCJmX1kiXSxbNiwzLCIiLDAseyJsZXZlbCI6Miwic3R5bGUiOnsiaGVhZCI6eyJuYW1lIjoibm9uZSJ9fX1dLFs4LDIsIiIsMix7ImxldmVsIjoyLCJzdHlsZSI6eyJoZWFkIjp7Im5hbWUiOiJub25lIn19fV0sWzgsNSwiIiwyLHsibGV2ZWwiOjIsInN0eWxlIjp7ImhlYWQiOnsibmFtZSI6Im5vbmUifX19XSxbOCw0LCJmX1kiLDJdLFszLDksIiIsMix7ImxldmVsIjoyLCJzdHlsZSI6eyJoZWFkIjp7Im5hbWUiOiJub25lIn19fV0sWzksNywiZl9ZIiwyXSxbNiw5LCIiLDIseyJsZXZlbCI6Miwic3R5bGUiOnsiaGVhZCI6eyJuYW1lIjoibm9uZSJ9fX1dXQ==
    \[\begin{tikzcd}[row sep=small]
        & Y && X && Y \\
        \pt && Y && Y && \pt \\
        & Y && Y && Y
        \arrow["{f_Y}"', from=1-2, to=2-1]
        \arrow[Rightarrow, no head, from=1-2, to=2-3]
        \arrow[Rightarrow, no head, from=1-2, to=3-2]
        \arrow["f"', from=1-4, to=2-3]
        \arrow["f", from=1-4, to=2-5]
        \arrow["f"{description}, from=1-4, to=3-4]
        \arrow[Rightarrow, no head, from=1-6, to=2-5]
        \arrow["{f_Y}", from=1-6, to=2-7]
        \arrow[Rightarrow, no head, from=1-6, to=3-6]
        \arrow[Rightarrow, no head, from=2-3, to=3-4]
        \arrow[Rightarrow, no head, from=2-5, to=3-6]
        \arrow["{f_Y}", from=3-2, to=2-1]
        \arrow[Rightarrow, no head, from=3-2, to=2-3]
        \arrow[Rightarrow, no head, from=3-4, to=2-5]
        \arrow["{f_Y}"', from=3-6, to=2-7]
    \end{tikzcd}\]
    We recall from \cref{span-adj} that the middle $2$-morphism is the counit $c^f_!$ in $\Spanoh(\cX)$.
    Thus, applying the lift $\CC\colon \Spanoh(\cX) \to \tCat$, we see that the our map is given by
    \[
        M_\CC[X] \simeq f_{X!} f_X^* M
        \simeq f_{Y!} f_! f^* f_Y^* M
        \too[c^f_!] f_{Y!} f_Y^* M
        \simeq M_\CC[Y],
    \]
    concluding the proof.
\end{proof}

We note that that aside from the functoriality in $M$ and $X$ which are evident from the construction, the construction is also manifestly functorial in $\CC$ itself along $\cX$-cocontinuous $\cX$-parameterized maps.
Stating this for a single morphism, we have the following.

\begin{prop}\label{hom-fun}
    Let $F\colon \CC \to \DD$ be an $\cX$-cocontinuous $\cX$-parameterized map between $\cX$-cocomplete $\cX$-parameterized categories.
    Then there is an isomorphism
    \[
        F_\pt(M_\CC[X]) \simeq (F_\pt(M))_\DD[X]
        \qin \DD_\pt
    \]
    natural in $M \in \CC_\pt$ and $X \in \cX$.
\end{prop}

\begin{proof}
    This is the commutativity of the following diagram
    % https://q.uiver.app/#q=WzAsNixbMCwwLCJcXENDX1xccHQiXSxbMCwxLCJcXEREX1xccHQiXSxbMSwwLCJcXEZ1bl57XFxjb2NwbH1fXFxjWChcXGNYX3svLX0sIFxcQ0MpIl0sWzIsMCwiXFxGdW4oXFxjWCwgXFxDQ19cXHB0KSJdLFsyLDEsIlxcRnVuKFxcY1gsIFxcRERfXFxwdCkiXSxbMSwxLCJcXEZ1bl57XFxjb2NwbH1fXFxjWChcXGNYX3svLX0sIFxcREQpIl0sWzAsMSwiRl9cXHB0IiwyXSxbMiwwLCJcXHNpbSIsMl0sWzIsM10sWzMsNCwiRl9cXHB0IFxcY2lyYyAtIl0sWzUsMSwiXFxzaW0iLDJdLFs1LDRdLFsyLDUsIkYiLDJdXQ==
    \[\begin{tikzcd}
        {\CC_\pt} & {\Fun^{\cocpl}_\cX(\cX_{/-}, \CC)} & {\Fun(\cX, \CC_\pt)} \\
        {\DD_\pt} & {\Fun^{\cocpl}_\cX(\cX_{/-}, \DD)} & {\Fun(\cX, \DD_\pt)}
        \arrow["{F_\pt}"', from=1-1, to=2-1]
        \arrow["\sim"', from=1-2, to=1-1]
        \arrow[from=1-2, to=1-3]
        \arrow["F"', from=1-2, to=2-2]
        \arrow["{F_\pt \circ -}", from=1-3, to=2-3]
        \arrow["\sim"', from=2-2, to=2-1]
        \arrow[from=2-2, to=2-3]
    \end{tikzcd}\]
\end{proof}

We recall that $\Spanoh(\cX)^{\co}$, the $2$-category with opposite $2$-morphisms of $\Spanoh(\cX)$, plays the dual role to $\Spanoh(\cX)$, i.e.\ it is the universal $\cX$-complete $\cX$-parameterized category.
Similarly (and as a consequence), $(\cX_{/-})^\op\colon \cX^\op \to \tCat$ is the free $\cX$-complete $\cX$-parameterized category.
This leads us to dually define $\cX$-parameterized cohomology, which enjoys the dual formal properties.

\begin{defn}
    Let $\CC\colon \cX^\op \to \tCat$ be an $\cX$-complete $\cX$-parameterized category.
    We define the \tdef{$\cX$-parameterized cohomology} functor to be the composition
    \[
        \CC_\pt
        \iso \Fun^{\cpl}_\cX((\cX_{/-})^\op, \CC)
        \too \Fun(\cX^\op, \CC_\pt),
        \qquad M \mapsto \mdef{M_\CC^{(-)}}.
    \]
\end{defn}

\subsection{Parameterized Integration}

In this subsection we study the case where $\CC$ is $\cX$-semiadditive, under the assumption that $\cX$ is truncated.
Under this assumption, the norm map provides an isomorphism $\Nm_f\colon f_! \iso f_*$.
In particular, this identifies $\cX$-parameterized homology and cohomology $M_\CC[X] \simeq M_\CC^X$, and endows $\cX$-parameterized cohomology with integration maps, in a way which we now make compatible with the right-way maps in a coherent way.

\begin{defn}\label{coh-int-def}
    Assume that $\cX$ is truncated.
    Let $\CC\colon \cX^\op \to \tCat$ be an $\cX$-semiadditive $\cX$-parameterized category.
    We define the \tdef{$\cX$-parameterized cohomology with integration} functor to be the composition
    \[
        \CC_\pt
        \iso \Fun^{\psa}_\cX(\Spano(\cX_{/-}), \CC)
        \too \Fun(\Spano(\cX), \CC_\pt),
        \qquad M \mapsto \mdef{M_\CC^{(-)}}
    \]
    of the inverse of the isomorphism from \cref{free-sa} with the functor restricting an $\cX$-cocontinuous $\cX$-parameterized functor to its value at $\pt \in \cX$.
    
    For a morphism $f\colon X \to Y$ in $\cX$, we call the image of the span $X \oot[f] Y \too Y$ under $M_\CC^{(-)}\colon \Spano(\cX) \to \CC_\pt$ \tdef{the integral along $f$} and denote it by
    \[
        \mdef{\int_f}\colon M_\CC^X \too M_\CC^Y.
    \]
\end{defn}

Note that, by construction, the restriction of $\cX$-parameterized cohomology with integration along the inclusions
\[
    \cX \too \Spano(\cX),
    \qquad \cX^\op \too \Spano(\cX)
\]
is the $\cX$-parameterized homology and cohomology respectively.
Thus, as expected, in this situation, for any $X \in \cX$, the norm map $\Nm_f$ provides an isomorphism
\[
    M_\CC[X] \simeq f_{X!} f_X^* M \simeq f_{X*} f_X^* M \simeq M_\CC^X.
\]
In the same way, the integration map is given by the composition
\[
    \int_f\colon M_\CC^X \simeq M_\CC[X] \too M_\CC[Y] \simeq M_\CC^Y,
\]
where the middle map is described in \cref{hom-counit} in terms of the counit $c^f_!$.

\begin{remark}
    Recall that $\cX$-parameterized homology can also be constructed using the universal property of $\Spanoh(\cX)$ from \cref{span-1.5-univ}.
    Under \cref{span-2-univ}, $\cX$-parameterized cohomology with integration admits a similar alternative construction.
\end{remark}

Similarly to the case of homology and cohomology from \cref{hom-fun}, this construction is also functorial in $\cX$-semiadditive $\cX$-parameterized maps.

\begin{prop}\label{coh-int-fun}
    Let $F\colon \CC \to \DD$ be an $\cX$-semiadditive $\cX$-parameterized map between $\cX$-semiadditive $\cX$-parameterized categories.
    Then there is an isomorphism
    \[
        F_\pt(M_\CC^X) \simeq (F_\pt(M))_\DD^X
        \qin \DD_\pt
    \]
    natural in $M \in \CC_\pt$ and $X \in \Spano(\cX)$.
    In particular, for any map $f\colon X \to Y$, there is a commutative diagram
    % https://q.uiver.app/#q=WzAsNCxbMCwwLCJGX1xccHQoTV9cXENDXlgpIl0sWzAsMSwiRl9cXHB0KE1fXFxDQ15ZKSJdLFsxLDAsIihGX1xccHQoTSkpX1xcREReWCJdLFsxLDEsIihGX1xccHQoTSkpX1xcREReWSJdLFswLDEsIkZfXFxwdChcXGludF9mKSIsMl0sWzIsMywiXFxpbnRfZiJdLFswLDIsIlxcc2ltIl0sWzEsMywiXFxzaW0iXV0=
    \[\begin{tikzcd}
        {F_\pt(M_\CC^X)} & {(F_\pt(M))_\DD^X} \\
        {F_\pt(M_\CC^Y)} & {(F_\pt(M))_\DD^Y}
        \arrow["\sim", from=1-1, to=1-2]
        \arrow["{F_\pt(\int_f)}"', from=1-1, to=2-1]
        \arrow["{\int_f}", from=1-2, to=2-2]
        \arrow["\sim", from=2-1, to=2-2]
    \end{tikzcd}\]
\end{prop}

\begin{proof}
    This is the commutativity of the following diagram
    % https://q.uiver.app/#q=WzAsNixbMCwwLCJcXENDX1xccHQiXSxbMCwxLCJcXEREX1xccHQiXSxbMSwwLCJcXEZ1bl57XFxwc2F9X1xcY1goXFxjWF97Ly19LCBcXENDKSJdLFsyLDAsIlxcRnVuKFxcY1gsIFxcQ0NfXFxwdCkiXSxbMiwxLCJcXEZ1bihcXGNYLCBcXEREX1xccHQpIl0sWzEsMSwiXFxGdW5ee1xccHNhfV9cXGNYKFxcY1hfey8tfSwgXFxERCkiXSxbMCwxLCJGX1xccHQiLDJdLFsyLDAsIlxcc2ltIiwyXSxbMiwzXSxbMyw0LCJGX1xccHQgXFxjaXJjIC0iXSxbNSwxLCJcXHNpbSIsMl0sWzUsNF0sWzIsNSwiRiIsMl1d
    \[\begin{tikzcd}
        {\CC_\pt} & {\Fun^{\psa}_\cX(\cX_{/-}, \CC)} & {\Fun(\cX, \CC_\pt)} \\
        {\DD_\pt} & {\Fun^{\psa}_\cX(\cX_{/-}, \DD)} & {\Fun(\cX, \DD_\pt)}
        \arrow["{F_\pt}"', from=1-1, to=2-1]
        \arrow["\sim"', from=1-2, to=1-1]
        \arrow[from=1-2, to=1-3]
        \arrow["F"', from=1-2, to=2-2]
        \arrow["{F_\pt \circ -}", from=1-3, to=2-3]
        \arrow["\sim"', from=2-2, to=2-1]
        \arrow[from=2-2, to=2-3]
    \end{tikzcd}\]
\end{proof}

\subsection{Parameterized Cardinalities}

We continue with the assumption that $\cX$ is truncated.
Let $\CC\colon \cX^\op \to \tCat$ be an $\cX$-semiadditive $\cX$-parameterized category, and consider the cohomology with integration functor from \cref{coh-int-def}
\[
    \CC_\pt \too \Fun(\Spano(\cX), \CC_\pt),
    \qquad M \mapsto M_\CC^{(-)}.
\]
Using the exponential adjunction we get a functor
\[
    \Spano(\cX) \too \End(\CC_\pt),
\]
which we may decategorify once more by considering its effect on endomorphisms of $\pt \in \Spano(\cX)$.
Recall that $\End_{\Spano(\cX)}(\pt) \simeq \cX^\simeq$, and the functor sends $\pt$ to $\Id_{\CC_\pt}$.

\begin{defn}
    We define the \tdef{$\cX$-parameterized cardinality} functor to be the functor
    \[
        \mdef{|-|_\CC}\colon \cX^\simeq \too \End(\Id_{\CC_\pt}).
    \]
    More concretely, this sends an object $X \in \cX$ to the composition
    \[
        |X|_\CC\colon \Id_{\CC_\pt} \too[u^f_*] f_{X*} f_X^* \too[\int_X] \Id_{\CC_\pt}.
    \]
    More concretely still, this sends $M \in \CC_\pt$ to the map
    \[
        |X|_\CC\colon M \too M_\CC^X \too[\int_X] M.
    \]
\end{defn}

This construction is once more functorial in $\cX$-semiadditive $\cX$-parameterized maps $F\colon \CC \to \DD$.

\begin{prop}\label{card-fun}
    Let $F\colon \CC \to \DD$ be an $\cX$-semiadditive $\cX$-parameterized map between $\cX$-semiadditive $\cX$-parameterized categories.
    Then the two maps
    % https://q.uiver.app/#q=WzAsMixbMCwwLCJGX1xccHQoTSkiXSxbMiwwLCJGX1xccHQoTSkiXSxbMCwxLCJGX1xccHQofFh8X1xcQ0MpIiwwLHsiY3VydmUiOi0yfV0sWzAsMSwifFh8X1xcREQiLDIseyJjdXJ2ZSI6Mn1dXQ==
    \[\begin{tikzcd}
        {F_\pt(M)} && {F_\pt(M)}
        \arrow["{F_\pt(|X|_\CC)}", curve={height=-12pt}, from=1-1, to=1-3]
        \arrow["{|X|_\DD}"', curve={height=12pt}, from=1-1, to=1-3]
    \end{tikzcd}\]
    are isomorphic naturally in $M \in \CC_\pt$ and $X \in \cX^\simeq$.
\end{prop}

\begin{proof}
    Recall the isomorphism from \cref{coh-int-fun}
    \[
        F_\pt(M_\CC^X) \simeq (F_\pt(M))_\DD^X
    \]
    which is natural in $M \in \CC_\pt$ and $X \in \Spano(\cX)$.
    Under the exponential adjunction, this can be views as the following commutative diagram
    % https://q.uiver.app/#q=WzAsNCxbMCwxLCJcXFNwYW5vKFxcY1gpIl0sWzIsMSwiXFxGdW4oXFxDQ19cXHB0LCBcXEREX1xccHQpIl0sWzEsMCwiXFxFbmQoXFxDQ19cXHB0KSJdLFsxLDIsIlxcRW5kKFxcRERfXFxwdCkiXSxbMCwyXSxbMiwxLCJGX1xccHQgXFxjaXJjIC0iXSxbMCwzXSxbMywxLCItIFxcY2lyYyBGX1xccHQiLDJdXQ==
    \[\begin{tikzcd}[row sep=small]
        & {\End(\CC_\pt)} \\
        {\Spano(\cX)} && {\Fun(\CC_\pt, \DD_\pt)} \\
        & {\End(\DD_\pt)}
        \arrow["{F_\pt \circ -}", from=1-2, to=2-3]
        \arrow[from=2-1, to=1-2]
        \arrow[from=2-1, to=3-2]
        \arrow["{- \circ F_\pt}"', from=3-2, to=2-3]
    \end{tikzcd}\]
    The result follows by considering the endomorphisms of $\pt \in \Spano(\cX)$.
\end{proof}

\subsection{Change of Parameterization}

In our applications, we will be interested in the interaction between change of parametrization and parameterized (co)homology and integration, a phenomenon which we now study.

\begin{notn}
    Let $\cX$ and $\cY$ be categories with pullbacks and terminal objects, and let $G\colon \cY \to \cX$ be a map preserving pullbacks and the terminal object.
\end{notn}

\begin{defn}
    We define the \tdef{change of parameterization along $G$} to be the functor given by pre-composition with $G$
    \[
        G^*\colon \tFun(\cX^\op, \tCat) \too \tFun(\cY^\op, \tCat).
    \]
\end{defn}

We begin with the case of parameterized homology.

\begin{prop}\label{pb-cc}
    The change of parameterization along $G$ restricts to parameterized cocomplete objects and parameterized cocontinuous maps, that is, it induces a functor
    \[
        G^*\colon \tFun^\cocpl(\cX^\op, \tCat) \too \tFun^\cocpl(\cY^\op, \tCat).
    \]
\end{prop}

\begin{proof}
    All conditions involved are about pullbacks in the source, which are preserved by $G$, and the image of morphisms, which are not changed.
\end{proof}

\begin{const}\label{G-cc-lift}
    We construct a lift of $G$ to a $\cY$-cocontinuous $\cY$-parameterized map
    \[
        G_{/-}\colon \cY_{/-} \too G^*(\cX_{/-}).
    \]

    As $G$ preserves pullbacks, we have an induced functor
    \[
        \Spanoh(G)\colon \Spanoh(\cY) \to \Spanoh(\cX).
    \]
    Consider $\pt \in \Spanoh(\cY)$, and recall that $G(\pt) \simeq \pt$, so we get an induced map
    \[
        \Hom_{\Spanoh(\cY)}(\pt, -)
        \too \Hom_{\Spanoh(\cX)}(G(\pt), G(-))
        \simeq G^*(\Hom_{\Spanoh(\cX)}(\pt, -))
    \]
    of functors $\Spanoh(\cY) \to \tCat$.
    Recall that as in the proof of \cref{free-cocpl}, the source is equivalent to $\cY_{/-}$ and the target to $G^*(\cX_{/-})$.
    Thus, the restriction along $\cY^\op \to \Spanoh(\cY)$ gives the desired functor, which is moreover $\cY$-cocontinuous by \cref{span-1.5-univ-fun} since it lifts to $\Spanoh(\cY)$ by construction.
\end{const}

\begin{remark}\label{G-cc-lift-other-way}
    There is another way to construct this functor.
    Recall from \cref{free-cocpl} that $\cY_{/-}$ is the free $\cY$-cocomplete $\cY$-parameterized category, thus the data of a $\cY$-cocontinuous $\cY$-parameterized map
    \[
        \cY_{/-} \too G^*(\cX_{/-})
    \]
    is equivalent to choosing an object in $\cX_{/G(\pt)} \simeq \cX$.
    The map constructed above clearly corresponds to $\pt \in \cX$.
    However, using this construction directly, it is unclear how to describe the corresponding map $\cY_{/-} \to G^*(\cX_{/-})$.
\end{remark}

We now show that the change of parameterization is compatible with parameterized homology.

\begin{prop}\label{G-cc-hom}
    Let $\CC\colon \cX^\op \to \tCat$ be an $\cX$-cocomplete $\cX$-parameterized category.
    Then there is an isomorphism
    \[
        M_{G^*(\CC)}[X] \simeq M_\CC[G(X)]
        \qin \CC_\pt
    \]
    natural in $M \in \CC_\pt$ and $X \in \cX$.
\end{prop}

\begin{proof}
    We need to show that the following diagram commutes
    % https://q.uiver.app/#q=WzAsNyxbMCwwLCJcXENDX1xccHQiXSxbMiwwLCJcXEZ1bihcXGNYLCBcXENDX1xccHQpIl0sWzAsMiwiXFxDQ19cXHB0Il0sWzIsMiwiXFxGdW4oXFxjWSwgXFxDQ19cXHB0KSJdLFsxLDAsIlxcRnVuXntcXGNvY3BsfV9cXGNYKFxcY1hfey8tfSwgXFxDQykiXSxbMSwyLCJcXEZ1bl57XFxjb2NwbH1fXFxjWShcXGNZX3svLX0sIEdeKihcXENDKSkiXSxbMSwxLCJcXEZ1bl57XFxjb2NwbH1fXFxjWShHXiooXFxjWF97Ly19KSwgR14qKFxcQ0MpKSJdLFswLDIsIiIsMSx7ImxldmVsIjoyLCJzdHlsZSI6eyJoZWFkIjp7Im5hbWUiOiJub25lIn19fV0sWzEsMywiR14qIl0sWzQsMCwiXFxzaW0iLDJdLFs0LDFdLFs1LDIsIlxcc2ltIiwyXSxbNSwzXSxbNCw2XSxbNiw1LCIoR197Ly19KV4qIiwyXV0=
    \[\begin{tikzcd}
        {\CC_\pt} & {\Fun^{\cocpl}_\cX(\cX_{/-}, \CC)} & {\Fun(\cX, \CC_\pt)} \\
        & {\Fun^{\cocpl}_\cY(G^*(\cX_{/-}), G^*(\CC))} \\
        {\CC_\pt} & {\Fun^{\cocpl}_\cY(\cY_{/-}, G^*(\CC))} & {\Fun(\cY, \CC_\pt)}
        \arrow[Rightarrow, no head, from=1-1, to=3-1]
        \arrow["\sim"', from=1-2, to=1-1]
        \arrow[from=1-2, to=1-3]
        \arrow[from=1-2, to=2-2]
        \arrow["{G^*}", from=1-3, to=3-3]
        \arrow["{(G_{/-})^*}"', from=2-2, to=3-2]
        \arrow["\sim"', from=3-2, to=3-1]
        \arrow[from=3-2, to=3-3]
    \end{tikzcd}\]
    The right (rectangle shaped) pentagon commutes because $G_{/-}$ is a lift of $G$.
    To see that the left (rectangle shaped) pentagon commutes, recall that horizontal morphisms are given by evaluation at $\pt \to \pt \in \cX_{/\pt}$ and $\pt \to \pt \in \cY_{/\pt}$, and since $G$ preserves the terminal object, the functor
    \[
        G_{/\pt}\colon \cY_{/\pt} \too G^*(\cX_{/\pt}) \iso \cX_{\pt}
    \]
    sends $\pt \to \pt$ to $\pt \to \pt$.
\end{proof}

The statements for parameterized cohomology hold dually, and move on to the case of parameterized semiadditivity.
The analogue of \cref{pb-cc} holds with the same proof (note that here we do not need truncatedness assumptions).

\begin{prop}\label{pb-sa}
    The change of parameterization along $G$ restricts to parameterized semiadditive objects and parameterized semiadditive maps, that is, it induces a functor
    \[
        G^*\colon \tFun^\psa(\cX^\op, \tCat) \too \tFun^\psa(\cY, \tCat).
    \]
\end{prop}

We now wish to construct $\Spano(G_{/-})$, the parameterized semiadditive analogue of \cref{G-cc-lift}.
We shall follow the same steps, using $\Spant$ in place of $\Spanoh$.
Note that this does not use \cref{span-2-univ}.
Rather, the construction itself is unrelated to parameterized semiadditivity, and to show that the resulting map is $\cY$-semiadditive, we only need the fact that $\cY^\op \to \Spant(\cY)$ is an example of a $\cY$-semiadditive $\cY$-parameterized object from \cref{span-2-sa}.

\begin{const}\label{G-psa-lift}
    Assume that $\cX$ and $\cY$ are truncated.
    We construct a $\cY$-semiadditive $\cY$-parameterized map
    \[
        \Spano(G_{/-})\colon \Spano(\cY_{/-}) \too G^*(\Span(\cX_{/-})).
    \]

    As $G$ preserves pullbacks, we have an induced functor
    \[
        \Spant(G)\colon \Spant(\cY) \to \Spant(\cX).
    \]
    
    Consider $\pt \in \Spant(\cY)$, and recall that $G(\pt) \simeq \pt$, so we get an induced map
    \[
        \Hom_{\Spant(\cY)}(\pt, -)
        \too \Hom_{\Spant(\cX)}(G(\pt), G(-))
        \simeq G^*(\Hom_{\Spant(\cX)}(\pt, -))
    \]
    of functors $\Spant(\cY) \to \tCat$.
    Recall that as we explain in \cref{free-sa-univ-sa}, the source is equivalent to $\Spano(\cY_{/-})$ and the target to $\Spano(G^*(\cX_{/-}))$.
    Thus, the restriction along $\cY^\op \to \Spant(\cY)$ gives the desired functor.
    Since $\cY^\op \to \Spant(\cY)$ is an example of a $\cY$-semiadditive $\cY$-parameterized object by \cref{span-2-sa}, the resulting map is $\cY$-semiadditive.
\end{const}

\begin{remark}
    As in \cref{G-cc-lift-other-way}, we can also construct the map via \cref{free-sa} by choosing $\pt \in \Spano(\cX)$.
    As in the previous case, it is unclear how to describe the corresponding map $\Spano(\cY_{/-}) \to G^*(\Spano(\cX_{/-}))$ from this construction.
\end{remark}

We now claim that the change of parameterization is compatible with parameterized cohomology with integration, as in the case of homology.

\begin{prop}\label{G-sa-cohint}
    Assume that $\cX$ and $\cY$ are truncated.
    Let $\CC\colon \cX^\op \to \tCat$ be an $\cX$-semiadditive $\cX$-parameterized category.
    Then there is an isomorphism
    \[
        M_{G^*(\CC)}^X \simeq M_\CC^{G(X)}
        \qin \CC_\pt
    \]
    natural in $M \in \CC_\pt$ and $X \in \Spano(\cX)$.
\end{prop}

\begin{proof}
    This is the exact same argument as in \cref{G-cc-hom}, employing $\Spano(G_{/-})$ from \cref{G-psa-lift} in place of $G_{/-}$ from \cref{G-cc-lift}.
\end{proof}

This also implies the compatibility of the change of parameterization with parameterized cardinalities.

\begin{prop}\label{G-sa-card}
    Assume that $\cX$ and $\cY$ are truncated.
    Let $\CC\colon \cX^\op \to \tCat$ be an $\cX$-semiadditive $\cX$-parameterized category.
    Then there is an isomorphism
    \[
        |X|_{G^*(\CC)} \simeq |G(X)|_\CC
        \qin \End(\Id_{\CC_\pt})
    \]
    natural in $X \in \cX^\simeq$.
\end{prop}

\begin{proof}
    Taking mate of the isomorphism from \cref{G-sa-cohint} under the exponential adjunction gives the following commutative diagram% https://q.uiver.app/#q=WzAsMyxbMCwwLCJcXFNwYW5vKFxcY1gpIl0sWzIsMCwiXFxFbmQoXFxDQ19cXHB0KSJdLFsxLDEsIlxcU3Bhbm8oXFxjWSkiXSxbMCwyLCJHXioiLDJdLFswLDFdLFsyLDFdXQ==
    \[\begin{tikzcd}
        {\Spano(\cX)} && {\End(\CC_\pt)} \\
        & {\Spano(\cY)}
        \arrow[from=1-1, to=1-3]
        \arrow["{G^*}"', from=1-1, to=2-2]
        \arrow[from=2-2, to=1-3]
    \end{tikzcd}\]
    where the unlabeled morphisms are the $\cX$-parameterized cohomology with integration functor associated to $\CC$ and the $\cY$-parameterized cohomology with integration functor associated to $G^*(\CC)$.
    The conclusion follows by considering the effect on endomorphisms of $\pt \in \Spano(\cX)$.
\end{proof}

	\section{Tempered Local Systems}\label{sec-temp}

We briefly recall Lurie's category of orbispaces.

\begin{defn}[{\cite[Definition 3.1.4]{Ell3}}]
    We denote by $\mdef{\TT}$ the full subcategory of $\Spaces$ spanned by classifying spaces $\BB H$ of finite abelian groups $H$.
    We call $\mdef{\OS} := \PSh(\TT^\op)$ the category of orbispaces.
    For an orbispace $X \in \OS$, we denote by $|X|$ its underlying space, given by the value at $\pt \in \TT$.
\end{defn}

The functor $X \mapsto |X|$ admits fully faithful adjoints from both sides.
We exclusively view spaces as fully faithfully embedded in orbispaces via the \emph{right} adjoint, sending $X \in \Spaces$ to the orbispace whose value at $T \in \TT$ is $\Map(T, X)$, which we denote by $X$ by a slight abuse of notation.

Let $R \in \CAlg(\Sp)$, let $\GG$ be an oriented $\PP$-divisible group over $R$, and let $X \in \OS$.
To this data, Lurie associates a category $\LocSys_\GG(X)$ of $\GG$-tempered local systems on $X$.
Lurie studies the functoriality of this construction in all three variables, and its relation to (non-tempered) local systems, all of which are central to the study of transchromatic characters.
In this section we extend the study of this construction in the parameter $X \in \OS$.
More specifically, we verify that the categories $\LocSys_\GG(X)$ assembles into an $\OS$-parameterized category, and that the functorialities in the two other variables, and the connection to (non-tempered) local systems, are $\OS$-parameterized as well, yielding $\OS$-parameterized adjunctions.

\subsection{Functorial Construction}

In this subsection we repeat the construction of tempered local systems from \cite{Ell3}, making the routine constructions exhibiting it as functorial in maps of commutative ring spectra and maps of orbispaces.
In particular, we view it as a functor
\[
    \LocSys\colon \PDivor \too \tFun(\OS^\op, \tCat),
\]
sending an oriented $\PP$-divisible group over a commutative ring spectrum to the $\OS$-parameterized category of tempered local systems.

We recall from \cite[Theorem 3.5.5]{Ell3}, that the data of an \emph{oriented $\PP$-divisible group} $\GG$ over $R$, is equivalent to the data of a functor
\[
    R_\GG\colon \TT^\op \too \CAlg(\Sp),
\]
satisfying certain conditions (being $\PP$-divisible in the sense of \cite[Definition 3.5.3]{Ell3}, and furthermore being oriented in the sense of \cite[\S2.5]{Ell3}).
Given an oriented $\PP$-divisible group $\GG$ over $R$ and an orbispace $X \in \OS$, Lurie constructs a category $\LocSys_\GG(X)$ of \emph{tempered local systems} on $X$, which is a certain full subcategory of $\Mod_{R_\GG}(\Fun(\TT_{/X}^\op, \Sp))$ (see \cite[Construction 5.1.3 and Definition 5.2.4]{Ell3}).

Note that a map of commutative ring spectra $R \to S$ gives rise to a map
\[
    R_\GG \too S \otimes_R R_\GG =: S_{\GG_S},
\]
where the tensor product on the right is computed point-wise, and $\GG_S$ is the base-change of $\GG$ to $S$.
This in turn gives rise to a base-change functor
\[
    \LocSys_\GG(X) \too \LocSys_{\GG_S}(X)
\]
given by tensoring with $S_{\GG_S}$ over $R_\GG$ (see \cite[Proposition 6.2.1]{Ell3}).

We begin by constructing the category of pairs of a commutative ring spectrum and an oriented $\PP$-divisible group over it.

\begin{defn}
    We let $\mdef{\PDivor} \subset \Fun(\TT^\op, \CAlg(\Sp))$ denote the subcategory whose objects are oriented $\PP$-divisible groups $R_\GG$, and whose morphisms are morphisms of the form above, namely coming from base-change along a map of commutative ring spectra.
\end{defn}

We now construct the functor of tempered local systems.

\begin{const}\label{ls-def}
    We construct the functor
    \[
        \mdef{\LocSys}\colon \PDivor \too \tFun(\OS^\op, \tCat).
    \]

    \underline{Step 1}:
    Recall that
    \[
        \Fun\colon \Cat^\op \times \Cat \too \Cat
    \]
    preserves limits in each variable.
    Thus, the mate under the exponential adjunction lands in product preserving functors
    \[
        \Fun\colon \Cat^\op \too \Fun^\times(\Cat, \Cat).
    \]
    A product preserving functor sends commutative monoids to commutative monoids, giving
    \[
        \Fun\colon \Cat^\op \too \Fun(\CMon(\Cat), \CMon(\Cat)),
    \]
    and we evaluate at $\Sp \in \CMon(\Cat)$ to get
    \[
        \Fun(-, \Sp)\colon \Cat^\op \too \CMon(\Cat),
    \]
    sending a category $\CC$ to $\Fun(\CC, \Sp)$ endowed with the point-wise symmetric monoidal structure.

    \underline{Step 2}:
    Consider the functor $\TT_{/(-)}^\op\colon \OS \to \Cat$ sending $X$ to $\TT_{/X}^\op$.
    Pre-composing this with the functor from the first step, we get
    \[
        \Fun(\TT_{/(-)}^\op, \Sp)\colon \OS^\op \too \CMon(\Cat).
    \]
    Observe that $\OS^\op$ has an initial object given by $\pt$, whence this functor lifts to commutative monoids under $\Fun(\TT^\op, \Sp)$.
    Since a commutative monoid under $\Fun(\TT^\op, \Sp)$ is in particular a module over it, we get a functor
    \begin{equation}\label{fun-to-mod}
        \Fun(\TT_{/(-)}^\op, \Sp)\colon \OS^\op \too \Mod_{\Fun(\TT^\op, \Sp)}(\Cat).
    \end{equation}

    \underline{Step 3}:
    Recall from \cite[Theorem 4.8.5.16]{HA} that for any presentably symmetric monoidal category $\CC$ we have a functor
    \[
        \Mod_{(-)}(-)\colon \CAlg(\CC) \times \Mod_\CC(\Cat) \too \Cat.
    \]
    Applying this to $\CC = \Fun(\TT^\op, \Sp)$ we get
    \begin{equation}\label{calg-mod}
        \Mod_{(-)}(-)\colon \Fun(\TT^\op, \CAlg(\Sp)) \times \Mod_{\Fun(\TT^\op, \Sp)}(\Cat) \too \Cat.
    \end{equation}

    \underline{Step 4}:
    Restrict \eqref{calg-mod} to $\PDivor \subset \Fun(\TT^\op, \CAlg(\Sp))$, and pre-compose with \eqref{fun-to-mod} to get
    \begin{equation}\label{mod-fun}
        \Mod_{(-)}(\Fun(\TT_{/(-)}^\op, \Sp))\colon \PDivor \times \OS^\op \too \Cat,
    \end{equation}
    sending $R_\GG$ and $X$ to $\Mod_{R_\GG}(\Fun(\TT_{/X}^\op, \Sp))$, with the functoriality of base-change along commutative ring spectra maps and pullback along maps of orbispaces.

    \underline{Step 5}:
    In \cite[Construction 5.1.3 and Definition 5.2.4]{Ell3}, Lurie constructs the full subcategory $\LocSys_\GG(X) \subset \Mod_{R_\GG}(\Fun(\TT_{/X}^\op, \Sp))$ of tempered local systems.
    Furthermore, by \cite[Proposition 6.2.1(d)]{Ell3}, base-change along commutative ring spectra maps sends tempered local systems to tempered local systems, and, similarly, by \cite[Remark 5.2.9]{Ell3} pullback along maps of orbispaces sends tempered local systems to tempered local systems.
    Therefore, using \cite[Proposition A.1]{Ram}, we get a subfunctor of \eqref{mod-fun} on the tempered local systems
    \[
        \LocSys\colon \PDivor \times \OS^\op \too \Cat,
    \]
    sending $R_\GG$ and $X$ to $\LocSys_\GG(X)$, with the functoriality of base-change along commutative ring spectra maps and pullback along maps of orbispaces.
    The desired functor is the mate of this under the exponential adjunction.
\end{const}

\subsection{Basic Properties}

The landmark result of \cite{Ell3} is the tempered ambidexterity theorem, saying that if $f\colon X \to Y$ is a relatively $\pi$-finite map of orbispaces, then the norm map is an isomorphism.
We rephrase this in our language.

\begin{prop}\label{ls-sa}
    Let $R \in \CAlg(\Sp)$ and let $\GG$ be an oriented $\PP$-divisible group over $R$.
    The $\OS$-parameterized category $\LocSys_\GG\colon \OS^\op \to \tCat$ is $\OS$-complete and $\OS$-cocomplete, and its restriction to $\Spacespi^\op \subset \OS^\op$ is $\Spacespi$-semiadditive.
\end{prop}

\begin{proof}
    For every map of orbispaces $f\colon X \to Y$ the map $f^*$ has left and right adjoints as is explained in \cite[Notation 7.0.1]{Ell3}.
    The Beck--Chevalley conditions for $f_*$ and $f_!$ are the contents of \cite[Theorem 7.1.6 and Corollary 7.1.7]{Ell3}.
    This shows the $\OS$-completeness cocompleteness.

    The $\Spacespi$-semiadditivity is a special case of \cite[Theorem 7.2.10]{Ell3}.
\end{proof}

Tempered local systems are also related to ordinary local systems.
Given an orbispace $X \in \OS$, there is an underlying space $|X| \in \Spaces$ given by the value at $\pt$.
Also recall that given an oriented $\PP$-divisible group $\GG$ over a commutative ring spectrum $R$, the value of $R_\GG$ at $\pt$ is $R$.
Using these observations, in \cite[Variant 5.1.15]{Ell3}, Lurie describes a functor
\[
    \LocSys_\GG(X) \too \Mod_R(\Sp)^{|X|}.
\]
We now show that this is natural in all arguments.

\begin{prop}\label{tempered-to-ordinary}
    There is a natural transformation
    \[
        \LocSys \too \Mod_{(-)}(\Sp)^{|-|}
    \]
    of functors $\PDivor \too \tFun(\OS^\op, \tCat)$.
\end{prop}

\begin{proof}
    Observe that we can construct the functor
    \[
        \Mod_{(-)}(\Sp)^{(-)}\colon \CAlg(\Sp) \times \Spaces^\op \too \Cat
    \]
    in a fashion similar to \cref{ls-def}.
    Indeed, repeating the construction with $\CAlg(\Sp)$ in place of $\PDivor$, $\Spaces$ in place of $\OS$, and $\{ \pt \}$ in place of $\TT$, we get
    \[
        \Mod_{(-)}(\Fun(\{ \pt \}_{/(-)}^\op, \Sp))\colon \CAlg(\Sp) \times \Spaces^\op \too \Cat,
    \]
    and since $\{ \pt \}_{/X}^\op \simeq X$ naturally in $X$, this is equivalent to the functor above.

    Now, pre-compose this with the functor $\PDivor \to \CAlg(\Sp)$ sending $R_\GG$ to $R$ and the functor $|-|\colon \OS \to \Spaces$.
    The inclusion $\{ \pt \} \subset \TT$ induces a natural map in the other direction
    \[
        \TT_{/(-)}^\op \too \{ \pt \}_{/(-)}^\op \simeq |-|,
    \]
    and evaluation at $\pt \in \TT$ sends $R_\GG$ to $R$.
    Thus, we get a natural transformation
    \[
        \Mod_{(-)}(\Fun(\TT_{/(-)}^\op, \Sp))
        \too \Mod_{(-)}(\Fun(\{ \pt \}_{/|-|}^\op, \Sp))
        \simeq \Mod_{(-)}(\Sp)^{|-|}.
    \]
    Finally, recall that $\LocSys$ is a subfunctor of the source, hence is a equipped with a natural transformation to it, and composing them we get the required natural transformation
    \[
        \LocSys \too \Mod_{(-)}(\Sp)^{|-|}.
    \]
\end{proof}

\subsection{Change of Rings}

In \cref{ls-def} we have constructed tempered local systems in a way that is functorial in the commutative ring spectrum, with the functoriality of base-change.
We now observe that this is an $\OS$-parameterized left adjoint.

\begin{prop}\label{base-change-sa}
    Let $\phi\colon R \to S$ be a map of commutative ring spectra and let $\GG$ be an oriented $\PP$-divisible group over $R$.
    There is an $\OS$-parameterized adjunction
    \[
        \phi^*\colon \LocSys_\GG \adj \LocSys_{\GG_S} \noloc \phi_*,
    \]
    where the left adjoint is given by base-change and the right adjoint is given by restriction of scalars.
\end{prop}

\begin{proof}
    By \cref{param-adj}, we need to check two conditions: first that for any $X \in \OS$ the map $\phi^*$ has a right adjoint, and second that the Beck--Chevalley condition holds.

    For the first part, recall that that at the level of modules, we have an adjunction
    \[
        \phi^*\colon \Mod_{R_\GG}(\Fun(\TT_{/X}^\op, \Sp)) \adj \Mod_{S_{\GG_S}}(\Fun(\TT_{/X}^\op, \Sp)) \noloc \phi_*.
    \]
    As we have seen earlier, \cite[Proposition 6.2.1(d)]{Ell3} says that the left adjoint restricts to tempered local systems, but \cite[Proposition 6.2.1(b)]{Ell3} says that the right adjoint restricts to tempered local systems as well, as required.

    For the second part, let $f\colon X \to Y$ be a map of orbispaces, and consider the commutative diagram
    % https://q.uiver.app/#q=WzAsNCxbMSwwLCJcXExvY1N5c19cXEdHKFgpIl0sWzAsMCwiXFxMb2NTeXNfXFxHRyhZKSJdLFswLDEsIlxcTG9jU3lzX3tcXEdHX1N9KFkpIl0sWzEsMSwiXFxMb2NTeXNfe1xcR0dfU30oWCkiXSxbMSwwLCJmXioiXSxbMiwzLCJmXioiXSxbMSwyLCJcXHBoaV4qIiwyXSxbMCwzLCJcXHBoaV4qIl1d
    \[\begin{tikzcd}
        {\LocSys_\GG(Y)} & {\LocSys_\GG(X)} \\
        {\LocSys_{\GG_S}(Y)} & {\LocSys_{\GG_S}(X)}
        \arrow["{f^*}", from=1-1, to=1-2]
        \arrow["{f^*}", from=2-1, to=2-2]
        \arrow["{\phi^*}"', from=1-1, to=2-1]
        \arrow["{\phi^*}", from=1-2, to=2-2]
    \end{tikzcd}\]
    We need to show that the Beck--Chevalley map
    \[
        f^* \phi_* \too \phi_* f^*
    \]
    is an isomorphism.
    This is clear at the level of modules, since all functors involved are given by restriction, and the conclusion follows since tempered local systems are full subcategories to which the adjunctions restrict.
\end{proof}

This can be rephrased more globally as follows.

\begin{cor}
    The functor of tempered local systems from \cref{ls-def} factors as
    \[
        \LocSys\colon \PDivor \too \tFun^{\cocpl}(\OS^\op, \tCat).
    \]
\end{cor}

\begin{proof}
    By \cref{ls-sa}, for any $R \in \CAlg(\Sp)$ and any $\GG \in \PDivor$, the $\OS$-parameterized category $\LocSys_\GG\colon \OS^\op \to \tCat$ is $\OS$-cocomplete.
    Next, by \cref{base-change-sa}, the base-change map is $\OS$-parameterized left adjoint, hence by \cref{pla-cond} it is in particular $\OS$-cocontinuous, concluding the proof.
\end{proof}

\subsection{Categorified Transchromatic Characters}

An essential part of the categorified transchromatic character map is the treatment of the case of an oriented $\PP$-divisible group with a constant direct summand, developed in \cite[\S6.4]{Ell3}.
To recall the main result, some definitions are in order.

\begin{defn}[{\cite[Definition 2.7.1]{Ell3}}]
    A \tdef{colattice} is a torsion abelian group $\Lambda$ such that for any $n > 0$ the map $n\colon \Lambda \to \Lambda$ is surjective with finite kernel.
\end{defn}

\begin{example}
    The group $(\Qp/\Zp)^h$ is a colattice for every $h \geq 0$.
\end{example}

\begin{defn}[{\cite[Construction 3.4.3]{Ell3}}]\label{fl}
    We denote by $\mdef{\fLLam}\colon \OS \to \OS$ the \emph{$\Lambda$-formal free loop space} functor, given by
    \[
        \fLLam X = \colim_{\Lambda_0 \subset \Lambda} \Hom_\OS(\BB \widehat{\Lambda}_0, X),
    \]
    where the colimit ranges over the finite subgroups $\Lambda_0 \subset \Lambda$, and is computed in the category of orbispaces, and $\widehat{\Lambda}_0$ denotes the Pontryagin dual.
\end{defn}

We recall from \cite[\S3.4]{Ell3} that for any space $X$, there is an associated comparison map
\[
    \fLLam X \too \Map(\BB \widehat{\Lambda}, X)
    \qin \OS
\]
which can informally be described as forgetting the adic topology on $\widehat{\Lambda}$.

\begin{warn}
    The comparison map is often not an equivalence (see \cite[Warning 3.4.9]{Ell3}), but is an an equivalence on $\pi$-finite spaces by \cite[Proposition 3.4.7]{Ell3}.
\end{warn}

\begin{example}\label{p-adic-fl}
    For the case $\Lambda = (\Qp/\Zp)^h$ the target of the comparison map is the \tdef{$h$-fold $p$-adic free loop space} $\mdef{\fL_p^h} := \Map(\BB \Zp^h, -)$.
\end{example}

As explained in \cite[\S6.4]{Ell3}, given an oriented $\PP$-divisible group $\GG$ over $R$, for each $X \in \OS$, there is an adjunction
\[
    \Phi\colon \LocSys_{\GG \oplus \Lambda}(X) \adj \LocSys_{\GG}(\fLLam X) \noloc \Psi,
\]
where $\GG \oplus \Lambda$ denotes the direct sum of $\GG$ with the constant oriented $\PP$-divisible group associated to $\Lambda$ (\cite[Construction 2.7.5]{Ell3}).
Our goal in this subsection is to make this construction into an $\OS$-parameterized adjunction.

We note that it is not a priori clear that either the construction of $\Phi$ or $\Psi$ from \cite[\S6.4]{Ell3} is functorial in $X$.
To pin point the difficulty, which makes their construction in loc.\ cit.\ somewhat cumbersome, observe that for $T \in \TT$, the $\Lambda$-formal free loop space $\fLLam T$ is often not in $\TT$.
Indeed, its set of connected components is isomorphic to $\Hom_\Ab(\widehat{\Lambda}, \pi_1(T))$ (see \cite[Notation 6.4.2]{Ell3}).
However, this is the only obstacle, and extending $\TT$ correspondingly makes it easier to make the constructions functorial.

\begin{defn}
    We let $\mdef{\Spacesoab}$ be the full subcategory of $\Spaces$ of $1$-finite spaces with abelian fundamental groups, that is, the category generated from $\TT$ under finite coproducts.
\end{defn}

\begin{prop}\label{fllam-res}
    The functor $\fLLam\colon \OS \to \OS$ restricts to a functor $\fLLam\colon \Spacesoab \to \Spacesoab$.
\end{prop}

\begin{proof}
    Let $X \in \Spacesoab$.
    As mentioned above, by \cite[Proposition 3.4.7]{Ell3}, the comparison map
    \[
        \fLLam X \too \Map(\BB \widehat{\Lambda}, X)
    \]
    is an isomorphism, so it remains to show that $\Map(\BB \widehat{\Lambda}, X)$ is $1$-finite and has abelian fundamental groups.

    Since $\BB \widehat{\Lambda}$ is connected, the functor $\Map(\BB \widehat{\Lambda}, -)$ commutes with finite coproducts, so we may assume that $X$ is connected, i.e.\ $X = \BB H$ for some finite abelian group $H$.
    In this case we can describe the mapping space as
    \[
        \Map(\BB \widehat{\Lambda}, \BB H) \simeq \Hom_\Ab(\widehat{\Lambda}, H) /\mkern-5mu/ H,
    \]
    where the $H$ acts by conjugation.
    Since $H$ is a finite abelian group, it suffices to show that $\Hom_\Ab(\widehat{\Lambda}, H)$ is a finite set.

    Let $n = |H|$ denote the order of $H$.
    Recall that by assumption $\Lambda[n]$ is finite.
    As the Pontryagin dual of the inclusion $\Lambda[n] \subset \Lambda$ is the quotient map $\widehat{\Lambda} \to \widehat{\Lambda}/n$, we see that quotient is finite.
    Finally, the quotient map induces an isomorphism
    \[
        \Hom_\Ab(\widehat{\Lambda}/n, H) \iso \Hom_\Ab(\widehat{\Lambda}, H),
    \]
    and the result follows since $\widehat{\Lambda}/n$ and $H$ are finite.
\end{proof}

Using these observations, we can give a more convenient model for tempered local systems.

\begin{prop}\label{ls-model}
    Right Kan extension defines an $\OS$-parameterized map
    \[
        \Mod_{R_\GG}(\Fun(\TT_{/(-)}^\op, \Sp)) \too \Mod_{R_\GG}(\Fun((\Spacesoab)_{/(-)}^\op, \Sp))
    \]
    which is level-wise fully faithful.
    In particular, $\LocSys_\GG$ is level-wise fully faithfully embedded in the target.
\end{prop}

\begin{proof}
    Repeat the construction from \cref{ls-def} with $\Spacesoab$ in place of $\TT$.
    Note that the inclusion $\TT \subset \Spacesoab$ is fully faithful, so the right Kan extensions
    \[
        \Fun(\TT_{/(-)}^\op, \Sp) \too \Fun((\Spacesoab)_{/(-)}^\op, \Sp)
    \]
    are symmetric monoidal and level-wise fully faithful.
    All of the constructions in \cref{ls-def} are tensored over the case of $\pt \in \OS$, hence given $R_\GG \in \Fun(\TT^\op, \CAlg(\Sp))$, we get fully faithful functors
    \[
        \Mod_{R_\GG}(\Fun(\TT_{/(-)}^\op, \Sp)) \too \Mod_{R_\GG}(\Fun((\Spacesoab)_{/(-)}^\op, \Sp))
    \]
    as required.
\end{proof}

With this model in mind, the construction of $\Psi$ becomes easy, and manifestly functorial in $X \in \OS$.
Recall from \cref{fllam-res} that $\fLLam\colon \OS \to \OS$ restricts to $\fLLam\colon \Spacesoab \to \Spacesoab$.
Therefore, we get an $\OS$-parameterized map
\[
    (\fLLam)^*\colon \Fun((\Spacesoab)_{/\fLLam(-)}^\op, \Sp) \too \Fun((\Spacesoab)_{/(-)}^\op, \Sp),
\]
given by
\[
    (\fLLam)^*(\lsF)(T \to X) = \lsF(\fLLam T \to \fLLam X),
\]
which is also evidently symmetric monoidal.
Also note that the for $X = \pt$, it sends $R_\GG$ to $R_{\GG \oplus \Lambda}$.
Thus, we get an induced $\OS$-parameterized map
\[
    (\fLLam)^*\colon \Mod_{R_\GG}(\Fun((\Spacesoab)_{/\fLLam(-)}^\op, \Sp)) \too \Mod_{R_{\GG \oplus \Lambda}}(\Fun((\Spacesoab)_{/(-)}^\op, \Sp)).
\]
By \cite[Proposition 6.4.6]{Ell3}, for every $X \in \OS$, this sends tempered local systems to tempered local systems under the inclusion from \cref{ls-model}.
Thus, by \cite[Proposition A.11]{Ram}, this gives rise to an $\OS$-parameterized map which we turn into a definition.

\begin{defn}
    We denote by
    \[
        \mdef{\Psi}\colon (\fLLam)^* \LocSys_{\GG} \too \LocSys_{\GG \oplus \Lambda}
    \]
    the subfunctor of $(\fLLam)^*$ restricted to tempered local systems.
\end{defn}

As we mentioned above, Lurie proves that this functor admits a left adjoint at every $X \in \OS$.
In proving that this assembles into a parameterized adjunction, we shall employ the formula for the left adjoint.
In order to state this formula, we prove the following lemma stated in \cite[Remark 6.4.11]{Ell3}.

\begin{lem}
    Let $T \in \TT$ and let $X \in \OS$.
    There is a canonical equivalence
    \[
        \Map(T, \fLLam X) \simeq \colim_{\Lambda_0 \subset \Lambda} \Map(T \times \BB \widehat{\Lambda}_0, X).
    \]
\end{lem}

\begin{proof}
    Recall that the identification of $\TT$ as a subcategory of $\OS$ is via the Yoneda embedding.
    Therefore, $\Map(T, -)$ commutes with colimits, so that
    \begin{align*}
        \Map(T, \fLLam X)
        &\simeq \Map(T, \colim_{\Lambda_0 \subset \Lambda} \Hom_\OS(\BB \widehat{\Lambda}_0, X))\\
        &\simeq \colim_{\Lambda_0 \subset \Lambda} \Map(T, \Hom_\OS(\BB \widehat{\Lambda}_0, X))\\
        &\simeq \colim_{\Lambda_0 \subset \Lambda} \Map(T \times \BB \widehat{\Lambda}_0, X).
    \end{align*}
\end{proof}

\begin{defn}\label{rep-by}
    Let $T \in \TT$ and let $X \in \OS$.
    We say that a map $T \to \fLLam X$ is \tdef{represented by} a map $T \times \BB \widehat{\Lambda}_0 \to X$ if it corresponds to a map in the colimit in the above lemma.
\end{defn}

\begin{prop}[{\cite[Theorem 6.4.9 and Example 6.4.12]{Ell3}}]\label{phi-form}
    For every $X \in \OS$, the functor $\Psi$ at $X$ admits a left adjoint
    \[
        \Phi\colon \LocSys_{\GG \oplus \Lambda}(X) \too \LocSys_{\GG}(\fLLam X).
    \]
    For $\lsF \in \LocSys_{\GG \oplus \Lambda}(X)$ and a map $T \to \fLLam X$ represented by a map $T \times \BB \widehat{\Lambda}_0 \to X$, we have
    \[
        \Phi(\lsF)(T \to \fLLam X)
        \simeq R_{\GG}^T \otimes_{R_{\GG}^{\fLLam(T \times \BB \widehat{\Lambda}_0)}} \lsF(T \times \BB \widehat{\Lambda}_0 \to X).
    \]
\end{prop}

We now show that the level-wise left adjoints assembles into an $\OS$-parameterized adjunction.

\begin{prop}\label{cat-chars-sa}
    There is an $\OS$-parameterized adjunction
    \[
        \Phi\colon \LocSys_{\GG \oplus \Lambda} \adj (\fLLam)^* \LocSys_{\GG} \noloc \Psi.
    \]
\end{prop}

\begin{proof}
    By \cref{param-adj} it remains to check that the Beck--Chevalley condition holds.
    Let $f\colon X \to Y$ be a map of orbispaces, and consider the commutative diagram
    % https://q.uiver.app/#q=WzAsNCxbMSwwLCJcXExvY1N5c197XFxHR30oXFxmTExhbSBYKSJdLFswLDAsIlxcTG9jU3lzX3tcXEdHfShcXGZMTGFtIFkpIl0sWzAsMSwiXFxMb2NTeXNfe1xcR0cgXFxvcGx1cyBcXExhbWJkYX0oWSkiXSxbMSwxLCJcXExvY1N5c197XFxHRyBcXG9wbHVzIFxcTGFtYmRhfShYKSJdLFsxLDAsIihcXGZMTGFtIGYpXioiXSxbMiwzLCJmXioiXSxbMSwyLCJcXFBzaSIsMl0sWzAsMywiXFxQc2kiXV0=
    \[\begin{tikzcd}
        {\LocSys_{\GG}(\fLLam Y)} & {\LocSys_{\GG}(\fLLam X)} \\
        {\LocSys_{\GG \oplus \Lambda}(Y)} & {\LocSys_{\GG \oplus \Lambda}(X)}
        \arrow["{(\fLLam f)^*}", from=1-1, to=1-2]
        \arrow["\Psi"', from=1-1, to=2-1]
        \arrow["\Psi", from=1-2, to=2-2]
        \arrow["{f^*}", from=2-1, to=2-2]
    \end{tikzcd}\]
    Our goal is to show that the Beck--Chevalley map associated to vertical left adjoints
    \[
        \Phi f^* \too (\fLLam f)^* \Phi
    \]
    is an isomorphism.
    To check that a natural transformation is an isomorphism, it suffices to check that it is an isomorphism on each object, namely that for every $\lsF \in \LocSys_{\GG \oplus \Lambda}(Y)$ the map
    \[  
        \Phi(f^*(\lsF)) \too (\fLLam f)^*(\Phi(\lsF))
    \]
    is an isomorphism.
    This, again, is a natural transformation, so it suffices to check that it is an isomorphism on each object $T \to \fLLam X$, namely that
    \[
        \Phi(f^*(\lsF))(T \to \fLLam X) \too (\fLLam f)^*(\Phi(\lsF))(T \to \fLLam X)
    \]
    is an isomorphism.
    
    The map $T \to \fLLam X$ is represented by some $T \times \BB \widehat{\Lambda}_0 \to X$.
    It follows that the composition $T \to \fLLam X \xrightarrow{\fLLam f} \fLLam Y$ is represented by $T \times \BB \widehat{\Lambda}_0 \to X \xrightarrow{f} Y$.
    Using the description of $\Phi$ from \cref{phi-form} we get
    \begin{align*}
        \Phi(f^*(\lsF))(T \to \fLLam X)
        &\simeq R_{\GG}^T \otimes_{R_{\GG}^{\fLLam(T \times \BB \widehat{\Lambda}_0)}} f^*(\lsF)(T \times \BB \widehat{\Lambda}_0 \to X)\\
        &\simeq R_{\GG}^T \otimes_{R_{\GG}^{\fLLam(T \times \BB \widehat{\Lambda}_0)}} \lsF(T \times \BB \widehat{\Lambda}_0 \to X \xrightarrow{f} Y)\\
        &\simeq \Phi(\lsF)(T \to \fLLam X \xrightarrow{\fLLam f} \fLLam Y)\\
        &\simeq (\fLLam f)^*(\Phi(\lsF))(T \to \fLLam X)
    \end{align*}
    as required.
\end{proof}

\subsection{The Local Case}

For any $t \geq 0$, we may consider the full subcategory
\[
    \LocSys_\GG^{\Kt}(X) \subset \LocSys_\GG(X)
\]
on those tempered local systems taking values in $\SpKt$.
Clearly, for $f\colon X \to Y$, the map $f^*\colon \LocSys_\GG(Y) \to \LocSys_\GG(X)$ sends a $\Kt$-local tempered local systems to a $\Kt$-local tempered local systems, thus by \cite[Proposition A.11]{Ram} we have an $\OS$-parameterized map
\[
    \LocSys_\GG^{\Kt} \too \LocSys_\GG.
\]
As explained in \cite[Notation 6.1.15]{Ell3}, for any $X \in \OS$, the inclusion admits a left adjoint $\LKt$, given by applying $\Kt$-localization level-wise, which we assemble over $\OS$.

\begin{prop}\label{local-tempered}
    There is an $\OS$-parameterized adjunction
    \[
        \LKt\colon \LocSys_\GG \adj \LocSys_\GG^{\Kt}\colon i.
    \]
\end{prop}

\begin{proof}
    As mentioned above, at every $X \in \OS$ there is an adjunction $\LKt \dashv i$.
    By \cref{param-adj}, it remains to check the Beck--Chevalley condition.
    Indeed, pullback along $f\colon X \to Y$ is given by pre-composition, while $\Kt$-localization is given by post-composition, so they clearly commute.
\end{proof}

Assume now that $R$ is $\Kn$-local for some $n \geq 0$, and that $\GG$ has height $n$ (that is, it is its Quillen $\PP$-divisible group).
In \cite[Theorem 6.3.1]{Ell3}, Lurie shows that the map
\[
    \LocSys_\GG(X) \too \Mod_R(\Sp)^{|X|}
\]
described in \cref{tempered-to-ordinary} restricts to an equivalence on $\Kn$-local objects.
We now phrase this in an $\OS$-parameterized manner.

\begin{defn}
    Let $R$ be a $\Kn$-local commutative ring spectrum.
    We denote by $\mdef{\cMod_R} := \Mod_R(\SpKn)$ the category of $\Kn$-local modules over $R$.
\end{defn}

\begin{prop}\label{kn-tempered-to-ordinary}
    Let $R$ be a $\Kn$-local commutative ring spectrum and let $\GG$ be its Quillen $\PP$-divisible group.
    The $\OS$-parameterized map from \cref{tempered-to-ordinary} restricts to an equivalence
    \[
        \LocSys_\GG^{\Kn} \iso \cMod_R^{(-)}.
    \]
\end{prop}

\begin{proof}
    First, for every $X \in \Spaces$ the map
    \[
        \LocSys_\GG(X) \too \Mod_R(\Sp)^{|X|}
    \]
    clearly lands in $\cMod_R^X \subset \Mod_R(\Sp)^X$.
    Thus, by \cite[Proposition A.11]{Ram}, the $\OS$-parameterized map restricts an $\OS$-parameterized map
    \[
        \LocSys_\GG^{\Kn} \too \cMod_R^{(-)}.
    \]
    Second, by \cite[Theorem 6.3.1]{Ell3}, the restricted map is an isomorphism at every $X \in \OS$, and the result follows since a natural transformation is an isomorphism if and only if it is an isomorphism at each object.
\end{proof}

	\section{Transchromatic Induced Characters}\label{sec-ind-char}

In this section we apply all the machinery developed in the previous sections to construct the categorified transchromatic character map coherently and establish its semiadditivity (\cref{main-thm}), deduce the compatibility of the transchromatic character map with induction (\cref{ind-char-form}), and calculate chromatic cardinalities (\cref{chrom-card}).

\begin{defn}
    Let $\mdef{\En} \in \CAlg(\SpKn)$ denote the Lubin--Tate spectrum.
    We denote by $\mdef{\GG} := \GG_{\En}^Q$ its Quillen oriented $p$-divisible group (see \cite[\S4.6]{Ell2}), and by $\mdef{\cMod_{\En}}$ the category of $\Kn$-local $\En$-modules.
\end{defn}

Let $0 \leq t \leq n$, and consider $\LKt\En \in \CAlg(\SpKt)$.
Base-changing $\GG$ along the localization $\En \to \LKt\En$, we get an oriented $p$-divisible group with an associated connected--\'etale short exact sequence (c.f.\ \cite[\S2.5]{Ell2})
\[
    0 \too \GG_{\LKt\En}^\circ \too \GG_{\LKt\En} \too \GG_{\LKt\En}^\et \too 0.
\]
The connected part $\GG_{\LKt\En}^\circ$ has height $t$ since it is isomorphic to the Quillen oriented $p$-divisible group of $\LKt\En$ (c.f.\ \cite[Remark 2.5.3]{Ell3}).
Therefore, the height of  the \'etale part $\GG_{\LKt\En}^\et$ is $n-t$.
As explained in \cite[Definition 2.7.12 and Proposition 2.7.15]{Ell3}, there is a universal commutative $\LKt\En$-algebra, called the splitting algebra, over which the short exact sequence splits and the \'etale part becomes constant on $\Lambda := (\Qp/\Zp)^{n-t}$.

\begin{defn}\label{ct-def}
    We denote by $\mdef{C_t} := \Split_\Lambda(\GG_{\LKt\En}^\circ \to \GG_{\LKt\En}) \in \CAlg(\SpKt)$ the splitting algebra, namely, the initial $\LKt\En$-algebra equipped with an isomorphism
    \[
        \GG_{C_t} \simeq \GG_{C_t}^Q \oplus \Lambda.
    \]
    We denote by $\mdef{\cMod_{C_t}}$ the category of $\Kt$-local $C_t$-modules.
\end{defn}

\begin{remark}
    We warn the reader that in \cite{Stapleton} Stapleton denotes by $C_t$ the non-spectral version of the splitting algebra considered here, which can be recovered by applying $\pi_0$.
\end{remark}

Associated to $\cMod_{\En}$ and $\cMod_{C_t}$ are the canonically $\Spaces$-parameterized categories $\cMod_{\En}^{(-)}$ and $\cMod_{C_t}^{(-)}$ from \cref{can-param}.
Pulling back along underlying functor $|-|\colon \OS \to \Spaces$ we get $\OS$-parameterized categories $\cMod_{\En}^{|-|}$ and $\cMod_{C_t}^{|-|}$.
We also pull back the latter along $\fLLam\colon \OS \to \OS$ from \cref{fl} to obtain the $\OS$-parameterized category  $\cMod_{C_t}^{|\fLLam(-)|}$.
With these in place, we can now construct the categorified transchromatic character map.

\begin{defn}\label{cch-def}
    We define the \tdef{categorified transchromatic character map}
    \[
        \cch\colon \cMod_{\En}^{|-|} \too \cMod_{C_t}^{|\fLLam(-)|}
        \qin \tFun(\OS^\op, \tCat)
    \]
    to be the composition of $\OS$-parameterized maps
    \begin{align*}
        \cMod_{\En}^{|-|}
        &\iso \LocSys_{\GG}^{\Kn}\\
        &\hooktoo \LocSys_{\GG}\\
        &\too \LocSys_{\GG_{C_t}}\\
        &\iso \LocSys_{\GG_{C_t}^Q \oplus \Lambda}\\
        &\too (\fLLam)^* \LocSys_{\GG_{C_t}^Q}\\
        &\too (\fLLam)^* \LocSys_{\GG_{C_t}^Q}^{\Kt}\\
        &\iso \cMod_{C_t}^{|\fLLam(-)|}
    \end{align*}
    where the maps are
    \begin{enumerate}
        \item the inverse of the isomorphism from \cref{kn-tempered-to-ordinary},
        \item the inclusion from \cref{local-tempered},
        \item the base-change along $\En \to C_t$ from \cref{base-change-sa},
        \item the isomorphism of oriented $p$-divisible groups from \cref{ct-def},
        \item the left adjoint from \cref{cat-chars-sa},
        \item the $\Kt$-localization from \cref{local-tempered},
        \item the isomorphism from \cref{kn-tempered-to-ordinary}.
    \end{enumerate}
\end{defn}

We now calculate the map in the case of $\pt \in \OS$.

\begin{prop}\label{cat-trans-char-pt}
    The map $\cch_\pt$ is equivalent to
    \[
        \LKt(C_t \otimes_{\En} -)\colon \cMod_{\En} \too \cMod_{C_t}.
    \]
\end{prop}

\begin{proof}
    Observe that the last map in the composition defining $\cch_\pt$ is given by evaluation at $\pt \iso \fLLam(\pt) \in \TT_{/\fLLam(\pt)}$.
    We shall prove that $\Phi$ from \cref{cat-chars-sa} does not change the value at the point, and conclude as all of the other maps involved are computed level-wise as in the statement of the proposition.
    Indeed, $\pt \to \fLLam(\pt)$ is, of course, represented by $\pt \to \pt$ (in the sense of \cref{rep-by}).
    Therefore, by \cref{phi-form}, for $\lsF \in \LocSys_{\GG_{C_t}^Q \oplus \Lambda}(\pt)$ we have
    \[
        \Phi(\lsF)(\pt \to \fLLam(\pt)) \simeq \Phi(\lsF)(\pt \to \pt),
    \]
    namely, it does not change the value at the point, and conclusion follows.
\end{proof}

Recall from \cite[Proposition 3.4.7]{Ell3} (see also \cref{p-adic-fl}) that for a $\pi$-finite space $X$ the comparison map
\[
    \fLLam X \too \fL_p^{n-t} X := \Map(\BB \Zp^h, X)
\]
is an isomorphism.
Therefore, upon restricting to $\Spacespi \subset \OS$, we may replace all instances of $\fLLam$ above by $\fL_p^{n-t}$.
With this in mind, we state our main theorem.

\begin{thm}\label{main-thm}
    The restriction of the categorified transchromatic character map to $\Spacespi$
    \[
        \cch\colon \cMod_{\En}^{(-)} \too \cMod_{C_t}^{\fL_p^{n-t}(-)}
    \]
    is a $\Spacespi$-semiadditive $\Spacespi$-parameterized map between $\Spacespi$-semiadditive $\Spacespi$-parameterized categories.
\end{thm}

\begin{proof}
    First, we claim that all $\Spacespi$-parameterized $\Spacespi$-categories involved in \cref{cch-def} are $\Spacespi$-semiadditive.
    The (non-parameterized) categories $\cMod_{\En}$ and $\cMod_{C_t}$ are $\infty$-semiadditive by \cite[Theorem 5.3.1]{TeleAmbi} (see also \cite[Theorem 5.2.1]{HL}).
    Hence, the associated canonically $\Spacespi$-parameterized categories are $\Spacespi$-semiadditive by \cref{can-param-sa}.
    All of the $\Spacespi$-parameterized categories of the form $\LocSys_{\GG'}$ where $\GG'$ is some oriented $p$-divisible group are $\Spacespi$-semiadditive by \cref{ls-sa}.
    Since $\fL_p^{n-t}$ preserves pullbacks, \cref{pb-sa} implies that applying $(\fL_p^{n-t})^*$ preserves $\Spacespi$-semiadditivity.
    Thus, all $\Spacespi$-parameterized categories involved are $\Spacespi$-semiadditive as required.
    
    Observe that each of the $\OS$-parameterized maps involved in the construction of $\cch$ is either a left or right adjoint (or even an equivalence).
    Thus, the restriction of each map to $\Spacespi \subset \OS$ is either $\Spacespi$-parameterized left or right adjoint, and since it is between $\Spacespi$-semiadditive categories it is $\Spacespi$-semiadditive by \cref{left-sa} (and it dual).
    Therefore, their composition is $\Spacespi$-semiadditive as well, concluding the proof.
\end{proof}

Using this and the decategorification process to parameterized cohomology with integration, we get the following.

\begin{thm}\label{ind-char-form}
    Let $0 \leq t \leq n$, then there is an isomorphism
    \[
        C_t \otimes_{\En} \En^X \simeq C_t^{\fL_p^{n-t} X}
        \qin \cMod_{C_t}
    \]
    natural in $X \in \Spano(\Spacespi)$.
\end{thm}

\begin{proof}
    To avoid excessive subscripts, we denote $\CC := \cMod_{\En}^{(-)}$, $\DD := \cMod_{C_t}^{(-)}$, and $\rmE := \En$.
    We apply our decategorification process:
    combining \cref{coh-int-fun} applied to $\cch$ with \cref{G-sa-cohint} applied to $\fL_p^{n-t}$, we get isomorphisms
    \[
        \cch_\pt(\rmE_\CC^X)
        \simeq (\cch_\pt(\rmE))_{(\fL_p^{n-t})^*(\DD)}^X
        \simeq (\cch_\pt(\rmE))_{\DD}^{\fL_p^{n-t} X}
        \qin \DD_\pt
    \]
    natural in $X \in \Spano(\Spacespi)$.
    Recall that since $\CC$ and $\DD$ are canonically $\Spacespi$-parameterized, the $\Spacespi$-parameterized cohomology is the ordinary cohomology as in \cref{hom-ord}.
    Combining this with the description of $\cch_\pt$ from \cref{cat-trans-char-pt}, the isomorphism above translates to a natural isomorphism
    \[
        \LKt(C_t \otimes_{\En} \En^X) \simeq C_t^{\fL_p^{n-t} X}
        \qin \cMod_{C_t}.
    \]
    
    To conclude, we need to show that the $\Kt$-localization on the left-hand side is redundant, i.e.\ that $C_t \otimes_{\En} \En^X$ is already $\Kt$-local for every $\pi$-finite space $X$.
    This follows from the non-functorial version of the isomorphism from the statement of the theorem, proven by Lurie \cite{Ell3}.
    Alternatively, recall from \cref{local-tempered} that the inclusion
    \[
        i\colon (\fLLam)^* \LocSys_{\GG_{C_t}^Q}^{\Kt} \hooktoo (\fLLam)^* \LocSys_{\GG_{C_t}^Q}
    \]
    is an $\OS$-parameterized right adjoint.
    As in the proof of \cref{main-thm}, it follows from \cref{left-sa} that the restriction of $i$ to $\Spacespi$ is $\Spacespi$-semiadditive.
    Thus, $i$ commutes with $f_X^*$ and $f_{X*}$, where $f_X\colon X \to \pt$ is the unique map.
    Since $C_t$ is $\Kt$-local, we get that $f_{X*} f_X^* C_t$ as computed in $(\fLLam)^* \LocSys_{\GG_{C_t}^Q}$ is already $\Kt$-local, rendering the $\Kt$-localization redundant in evaluating $\cch_\pt(\En^X)$.
\end{proof}

We also apply this to calculate the chromatic cardinalities.
Recall from \cite[Under Example 2]{card} that since the $\pi_0$ of the unit map $\SS_{\Kn} \to \En$ lands in $\Zp \subset \pi_0 \En$, the cardinality $|X|_{\En}$ is a $p$-adic integer, so although the chromatic cardinalities at different heights are a priori in different rings, the following statement makes sense.

\begin{thm}\label{chrom-card}
    For any $\pi$-finite space $X$ we have
    \[
        |X|_{\En} = |\fL_p^n X|_\QQ
        \qin \ZZ_{(p)}.
    \]
\end{thm}

\begin{proof}
    Consider the categorified transchromatic character map $\cch$ going to height $t = 0$
    \[
        \cch\colon \cMod_{\En}^{(-)} \too \cMod_{C_0}^{\fL_p^n(-)},
    \]
    which is $\Spacespi$-semiadditive by \cref{main-thm}.
    Applying \cref{card-fun}, we get
    \[
        \cch_\pt(|X|_{\En}) = |X|_{\cMod_{C_0}^{\fL_p^n(-)}}
        \qin \pi_0(C_0).
    \]
    Recall that by \cref{G-sa-card} we have
    \[
        |\fL_p^n X|_{C_0}
        := |\fL_p^n X|_{\cMod_{C_0}}
        = |X|_{\cMod_{C_0}^{\fL_p^n(-)}}
        \qin \pi_0(C_0).
    \]
    Combining the two, we get
    \[
        \cch_\pt(|X|_{\En}) = |\fL_p^n X|_{C_0}
        \qin \pi_0(C_0).
    \]
    Note that the restriction of the map $\pi_0(\En) \to \pi_0(C_0)$ to $\Zp \subset \pi_0(\En)$ is injective (for example since $\pi_0(C_0)$ is faithfully flat over $p^{-1}\En$ by \cite[Corollary 6.8]{HKR}).
    Therefore, we get an equality of $p$-adic integers
    \[
        |X|_{\En} = |\fL_p^n X|_{C_0}
        \qin \Zp.
    \]
    
    Consider the unit map $\QQ \to C_0$, which is injective on $\pi_0$.
    As above, we get an equality of rational numbers
    \[
        |\fL_p^n X|_\QQ = |\fL_p^n X|_{C_0}
        \qin \QQ
    \]
    
    Combining the results, we see that the two cardinalities in question are equal, and they are rational $p$-adic integers, i.e.\ in $\ZZ_{(p)}$, as required.
\end{proof}

	\bibliographystyle{alpha}
	\bibliography{refs}

\end{document}